\documentclass[11pt,letterpaper]{amsart}
\usepackage{amsmath, amsfonts, amsthm, amssymb, color}
\usepackage{graphicx}
\usepackage{float}
\usepackage{verbatim}
\usepackage{mathtools}

\usepackage{microtype}
\usepackage[T1]{fontenc}
\usepackage[utf8]{inputenc}
\usepackage{babel}
\usepackage{lmodern}
\usepackage{caption}
\usepackage{subcaption}

\usepackage[nocompress]{cite}

\allowdisplaybreaks

\numberwithin{equation}{section}

\newtheorem{lemma}{Lemma}[section]

\newtheorem{prop}[lemma]{Proposition}
\newtheorem{thm}[lemma]{Theorem}
\newtheorem{cor}[lemma]{Corollary}
\theoremstyle{definition}
\newtheorem{definition}[lemma]{Definition}

\newtheorem{conjecture}[lemma]{Conjecture}

\newtheorem{question}[lemma]{Question}
\theoremstyle{remark}
\newtheorem{remark}[lemma]{Remark}

\renewcommand{\epsilon}{\varepsilon}
\newcommand{\Good}{\mathit{Good}}
\newcommand{\Bad}{\mathit{Bad}}
\newcommand{\supp}{\mathrm{supp}}
\newcommand{\Var}{\mathrm{Var}}

\def\R{\mathbb{R}}
\def\C{\mathbb{C}}
\def\N{\mathbb{N}}

\def\L{\mathcal{L}}
\def\O{\mathcal{O}}
\def\a{\mathbf{a}}
\def\b{\mathbf{b}}

\def\j{\mathbf{j}}

\def\p{\mathbf{p}}

\def\A{\mathcal{A}}
\def\I{\mathcal{I}}
\def\O{\mathcal{O}}

\def\x{\mathbf{x}}
\def\p{\mathbf{p}}

\numberwithin{equation}{section} 
\numberwithin{table}{section}

\title[Fourier decay for fractal measures]{Polynomial Fourier decay for fractal measures and their pushforwards}

\author{Simon Baker and Amlan Banaji}

\date{}

\begin{document}
	
	\begin{abstract}
	We prove that the pushforwards of a very general class of fractal measures $\mu$ on $\mathbb{R}^d$ under a large family of non-linear maps $F \colon \mathbb{R}^d \to \mathbb{R}$ exhibit polynomial Fourier decay: there exist $C,\eta>0$ such that $|\widehat{F\mu}(\xi)|\leq C|\xi|^{-\eta}$ for all $\xi\neq 0$. 
	Using this, we prove that if $\Phi = \{ \varphi_a \colon [0,1] \to [0,1]\}_{a \in \mathcal{A}}$ is an iterated function system consisting of analytic contractions, and there exists $a \in \mathcal{A}$ such that $\varphi_a$ is not an affine map, then every non-atomic self-conformal measure for $\Phi$ has polynomial Fourier decay; this result was obtained simultaneously by Algom, Rodriguez Hertz, and Wang. 
	We prove applications related to the Fourier uniqueness problem, Fractal Uncertainty Principles, Fourier restriction estimates, and quantitative equidistribution properties of numbers in fractal sets.
	\end{abstract}
	
	\maketitle
	
	
		
	\section{Introduction}
	
	\subsection{Background and results}
	
	The \emph{Fourier transform} of a Borel probability measure $\mu$ supported on $\mathbb{R}$ is the function $\widehat{\mu} \colon \R \to \C$ given by 
	\begin{equation} 
		\label{e:1dFT}
		\widehat{\mu}(\xi) = \int_{\R} e^{-2\pi i \xi x} d\mu(x). 
	\end{equation}
	It is an important quantity giving `arithmetic' information about the measure. However, it is often difficult to calculate.
	A measure $\mu$ is said to be a \emph{Rajchman} measure if $\widehat{\mu}(\xi) \to 0$ as $|\xi| \to \infty$. Determining whether a measure is Rajchman, and if it is Rajchman, the speed at which the Fourier transform converges to zero, is an interesting and important problem. 
	
	In this paper we consider this problem in the context of fractal measures. 
	Historically, the study of the Fourier transform of fractal measures was initiated by problems coming from uniqueness of trigonometric series, metric number theory,
	Fourier multipliers, and maximal operators defined by fractal measures. We refer the reader to \cite{JialunSahlsten1,JialunSahlsten2,SahlstenSurvey} for a thorough historical overview. Well-studied families of fractal measures include self-similar measures and self-conformal measures; these arise from iterated function systems, which are defined as follows. We call a map $\varphi\colon [0,1]^{d}\to[0,1]^{d}$ a contraction if there exists $r\in(0,1)$ such that $\|\varphi(x)-\varphi(y)\|\leq r\|x-y\|$ for all $x,y\in [0,1]^d$. 
	We call a finite set of contractions an iterated function system, or IFS for short. 
	A well-known result due to Hutchinson~\cite{Hut} states that for any IFS $\Phi=\{\varphi_{a}\}_{a\in \A}$, there exists a unique non-empty compact set $X$ satisfying 
	\[ X=\bigcup_{a\in \A}\varphi_{a}(X).\] 
	We call $X$ the attractor of $\Phi$. We will always assume that our IFS is non-trivial, by which we mean that the contractions do not all share a common fixed point, and thus $X$ is always uncountable. 
	The attractors of iterated function systems very often exhibit fractal behaviour. 
	To understand the metric properties of an attractor $X$, one typically studies measures supported on $X$. The most well studied fractal measures are the stationary measures arising from probability vectors (see \cite{Hoc,Shmerkinsurvey,Varjusurvey2}).
	Given an IFS $\Phi=\{\varphi_{a}\}_{a\in \A}$ and a probability vector $\p=(p_a)_{a\in \A}$ ($0<p_a<1$ for all $a$ and $\sum_{a\in \A}p_a=1$), another well-known result, also due to Hutchinson \cite{Hut}, states that there exists a unique Borel probability measure satisfying 
	\begin{equation}\label{e:definebernoullimeas} 
	\mu=\sum_{a\in \A}p_a\cdot \varphi_a \mu. 
	\end{equation}
	 We call $\mu$ the stationary measure corresponding to $\Phi$ and $\p$, and emphasise that we always assume that all probability weights are strictly positive, so $\supp (\mu) = X$. 
	 Recall that a map $\varphi\colon [0,1]^{d}\to[0,1]^{d}$ is called a similarity if there exists $r \in (0,1)$ such that $\|\varphi(x)-\varphi(y)\|= r\|x-y\|$ for all $x,y\in [0,1]^d$. 
	 If an IFS $\Phi$ consists of similarities, and we want to emphasise this property, then we say that the IFS is a self-similar iterated function system, and the corresponding stationary measures are known as self-similar measures. 
	 Similarly, if $\Phi$ consists of $C^{1+\alpha}$ angle-preserving contractions with non-vanishing derivative, then the IFS is called a self-conformal IFS, and the corresponding stationary measures are known as self-conformal measures\footnote{In this paper, whenever we refer to self-conformal measures they will be supported in the line, so for our purposes the angle preserving property is superfluous.}. Self-conformal IFSs and self-conformal measures arise naturally in several areas of mathematics. 
	 They appear in number theory by considering the inverse branches of the Gauss map \cite{JordanSahlsten}. Similarly, the Furstenberg measures that play an important role in random matrix theory can under suitable hypothesis be realised as self-conformal measures~\cite{ABY,Yoccoz}. 
	
	Understanding the Fourier decay properties of stationary measures for IFSs is an active area of research; see~\cite{SahlstenSurvey} for a recent survey, which also discusses several applications of this problem. In this paper we study the Fourier decay properties of non-linear pushforwards of self-similar measures. This problem has been studied previously by Kaufman~\cite{Kau}, Chang and Gao~\cite{Chang2017Fourier}, and Mosquera and Shmerkin~\cite{MS}. They each proved polynomial Fourier decay results for pushforwards of self-similar measures coming from homogeneous self-similar IFSs in $\mathbb{R}$ (i.e. where all contraction ratios are equal) by $C^2$ maps with nonvanishing second derivative. These results were recently generalised to $\mathbb{R}^{2}$ by Mosquera and Olivo~\cite{MO}. Our main result in this direction is the following theorem, which makes no assumption of homogeneity and allows the second derivative of the pushforward map to vanish in places. 
    A measure $\nu$ is said to have polynomial Fourier decay (or power Fourier decay) if there exist $C,\eta > 0$ such that $|\widehat{\nu}(\xi)| \leq C|\xi|^{-\eta}$ for all $\xi \in \R \setminus \{0\}$. 
  \begin{thm}
        \label{thm:Analytic pushforward thm}
        Let $\mu$ be a self-similar measure with support in $[0,1]$ and let $F\colon [0,1]\to \mathbb{R}$ be analytic and non-affine. Then the pushforward measure $F\mu$ has polynomial Fourier decay. 
    \end{thm}
We will also prove statements for more general stationary measures on $\R^d$ and for pushforwards that are only $C^{2}$ instead of analytic (see Section~\ref{Section:Non-linear pushforwards of fractal measures}). Theorem~\ref{thm:Analytic pushforward thm} has an immediate consequence for the Fourier decay properties of certain self-conformal measures. Before formulating this statement we will recall some recent results on these measures.

It is generally believed that if a $C^{1+\alpha}$ IFS is sufficiently non-linear then each of its self-conformal measures will exhibit polynomial Fourier decay. This line of research was taken up by Jordan and Sahlsten \cite{JordanSahlsten} who, building upon work of Kaufman~\cite{Kau2}, and later Queff\'elec and Ramar\'{e}~\cite{QR}, proved that if $\mu$ is a Gibbs measure for the Gauss map of sufficiently large dimension then the Fourier transform of $\mu$ has polynomial Fourier decay. This result was then generalised by Bourgain and Dyatlov in~\cite{BD1}, who used methods from additive combinatorics to establish polynomial Fourier decay for Patterson--Sullivan measures for convex cocompact Fuchsian groups. The additive combinatorics methods of Bourgain and Dyatlov were later applied by Sahlsten and Stevens in~\cite{SS}. In this paper they proved that Gibbs measures for well separated self-conformal IFSs acting on the line satisfying a suitable non-linearity assumption have polynomial Fourier decay. Algom, Rodriguez Hertz, and Wang~\cite{AHW,AHW2} established weaker decay rates for the Fourier transform of self-conformal measures under weaker assumptions than those appearing in~\cite{SS}. Recently both the first named author and Sahlsten \cite{BS}, and Algom, Rodriguez Hertz, and Wang~\cite{AHW3}, gave sufficient conditions for an IFS which guarantee that every self-conformal measure exhibits polynomial Fourier decay. Both of these papers used a disintegration method inspired by work of Algom, the first named author, and Shmerkin~\cite{ABS}, albeit in different ways. Importantly, the results in \cite{AHW,AHW2,AHW3,BS} do not require the IFS to satisfy any separation assumptions. 
	
Building upon existing work, in this paper we use Theorem \ref{thm:Analytic pushforward thm} to prove that a self-conformal measure coming from an IFS consisting of analytic maps\footnote{Recall that a function $f\colon [0,1] \to \R$ is real analytic if for all $x_0 \in [0,1]$ there exists $\epsilon > 0$ and a power series about $x_0$ which converges to $f(x)$ for all $x \in (x_0 - \epsilon, x_0 + \epsilon) \cap [0,1]$.} will have polynomial Fourier decay under the weakest possible non-linearity assumption. In particular, the following holds. 	
	\begin{thm}
		\label{thm:analyticthm}
		Let $\{\varphi_a\colon [0,1] \to [0,1]\}_{a\in \A}$ be an IFS such that each $\varphi_{a}$ is analytic, and suppose that there exists $a \in \mathcal{A}$ such that $\varphi_{a}$ is not an affine map. Then for every self-conformal measure $\mu$ there exist $C,\eta>0$ such that $|\widehat{\mu}(\xi)|\leq C |\xi|^{-\eta}$ for all $\xi\neq 0$.
	\end{thm}
	
	Theorem~\ref{thm:analyticthm} was obtained simultaneously and independently by Algom et al. in \cite{AlgomPreprintnonlinear}. This theorem was announced in \cite{AHW3} and in \cite[Section~6]{AHW3} a short argument, conditional on what was at the time a forthcoming result\footnote{Preprints of the present paper and \cite{AlgomPreprintnonlinear} were made available to experts at the time \cite{AHW3} was made publicly available.} of Algom et al. \cite{AlgomPreprintnonlinear} (which also follows from this paper), was provided.

	The proof of Theorem~\ref{thm:analyticthm} divides into two cases depending on whether or not the IFS admits an analytic conjugacy to a self-similar IFS; both cases are highly non-trivial. The case where no such conjugacy exists follows from the main results in \cite{AHW3} and \cite{BS}, which are proved by establishing spectral gap-type estimates for appropriate transfer operators. 
For the case when the IFS admits such a conjugacy, each self-conformal measure can be realised as a pushforward of a self-similar measures under the conjugacy map. Using this observation, the conjugacy case of Theorem~\ref{thm:analyticthm} follows from a suitable application of Theorem \ref{thm:Analytic pushforward thm} (see the argument given in Section~\ref{s:analyticthm}). Similarly, a suitable theorem on the behaviour of non-linear pushforwards of self-conformal measures was obtained by Algom et al. in \cite{AlgomPreprintnonlinear}. It was this result that was relied upon in the conditional argument given in \cite{AHW3}. 

The present paper and \cite{AlgomPreprintnonlinear} differ in two key ways. Whereas \cite{AHW3}, \cite{BS} and this paper made use of a disintegration technique from \cite{ABS} to establish Fourier decay, Algom et al. \cite{AlgomPreprintnonlinear} use a large deviations estimate of Tsujii~\cite{Tsujii}. 
Our non-linear projection theorem applies more generally and covers some higher-dimensional and infinite iterated function systems. On the other hand, the non-linear projection theorem of Algom et al. is a more direct analogue of the classical van der Corput inequality from harmonic analysis. They also use this theorem to prove equidistribution results for the sequence $(x^n \mod 1)_{n=1}^{\infty}$, where $x$ is distributed according to a self-similar measure.
	
	\subsection{Non-linear pushforwards of fractal measures}
    \label{Section:Non-linear pushforwards of fractal measures}
	Our results on non-linear pushforwards of fractal measures hold in the general setting of countable IFSs, which were studied systematically in~\cite{MU1}. 
	As such it is necessary to introduce some additional terminology. 
	We call a countable family $\Phi=\{\varphi_{a} \colon [0,1]^d \to [0,1]^d \}_{a\in \A}$ of contractions satisfying 
    \begin{equation}\label{e:uniformcontraction}
    \sup_{a\in \A}\sup_{x,y\in[0,1]^{d}}\frac{\|\varphi_{a}(x)-\varphi_{a}(y)\|}{\|x-y\|}<1 
    \end{equation}
	a countable iterated function system, or CIFS for short. Whenever we use the phrase IFS we will mean a non-trivial finite set of contractions; when we say CIFS we will mean a countable (possibly finite, possibly trivial) collection of contractions satisfying~\eqref{e:uniformcontraction}. For a countable iterated function system there no longer necessarily exists a unique non-empty compact set $X$ satisfying $X=\cup_{a\in \A}\varphi_{a}(X)$.\footnote{The well-known fixed point proof no longer works because the map $X \mapsto \cup_{a\in \A}\varphi_{a}(X)$ does not necessarily map compact sets to compact sets.}
	As such, given a CIFS we define 
	\[ X\coloneqq \bigcup_{(a_i)_{i=1}^{\infty} \in \mathcal{A}^{\N}} \bigcap_{n=1}^{\infty} (\varphi_{a_1}\circ\cdots \circ \varphi_{a_n}) ([0,1]^d),\]
	and call $X$ the \emph{attractor} of $\Phi$. When our countable IFS contains finitely many contractions, i.e. it is an IFS, then the attractor as defined above coincides with the unique non-empty compact set satisfying $X=\cup_{a\in \A}\varphi_{a}(X)$, so there is no ambiguity in our use of the term attractor. Given a CIFS $\Phi=\{\varphi_{a}\}_{a\in \A}$ and a probability vector $\p=(p_a)_{a\in \A}$ there exists a unique Borel probability measure satisfying\footnote{The fixed point proof works in this context because $\mu \mapsto \sum_{a\in \A}p_a\varphi_a \mu$ is a map from the space of Borel probability measures supported on $[0,1]^d$ to itself, see \cite[Theorem~2]{Secelean2002countable}.} %
	\[ \mu=\sum_{a\in \A}p_a\cdot \varphi_a \mu.\] 
	We call $\mu$ the stationary measure for $\Phi$ and $\p$. 
	
	In this paper we will focus on the stationary measures coming from the following special class of CIFSs. 
	
	\begin{definition}
    \label{Def:Fibre product}
		Let $\Psi \coloneqq \{\psi_{j} \colon [0,1]^{d}\to [0,1]^d\}_{j\in J}$ be a CIFS. Suppose that for each $j\in J$ there exists a CIFS $\Psi_{j} \coloneqq \{\gamma_{l,j} \colon [0,1]\to [0,1]\}_{l\in L_j}$ consisting of similarities, i.e. each $\Psi_{j}$ has the form $\Psi_{j}=\{\gamma_{l,j}(x)=r_{l,j}\cdot x + t_{l,j}\}_{l\in L_j}$ where $|r_{l,j}| \in (0,1)$ and $t_{l,j} \in \R$. Also, suppose that there exists $\Psi_{j^*}$ such that the corresponding attractor is not a singleton. Then we define the fibre product CIFS (consisting of maps from $[0,1]^{d+1}$ to itself) to be 
		\[\Phi=\left\{\varphi_{j,l}(x_1,\ldots,x_{d+1})=(\psi_{j}(x_1,\ldots,x_d),\gamma_{l,j}(x_{d+1}))\right\}_{j\in J,l\in L_j}.\] 
		We refer to $\Psi$ as the base CIFS, and to each $\Psi_{j}$ as a fibre CIFS. We will always assume that the fibre product CIFS is itself a CIFS, so satisfies the condition 
		\[ \sup_{j,l}\sup_{x,y\in[0,1]^{d+1}}\frac{\|\varphi_{j,l}(x)-\varphi_{j,l}(y)\|}{\|x-y\|}<1.\]
	\end{definition}
Examples of CIFSs that can be realised as fibre product CIFSs arise in the study of self-similar or self-affine carpets and sponges such as those of Bedford--McMullen~\cite{Bed,MC}, Gatzouras--Lalley~\cite{GL}, or Bara\'nski~\cite{Bar} type. Some more detailed examples are given in Figure~\ref{f:ifs}. 

	Given a $C^{2}$ function $F\colon [0,1]^{d+1}\to \mathbb{R}$, we associate the quantities 
	\[\|F\|_{\infty,1}\coloneqq \max_{x\in [0,1]^{d+1}}\left|\frac{\partial F}{\partial x_{d+1}}(x)\right|,\quad\|F\|_{\infty,2}\coloneqq \max_{x\in [0,1]^{d+1}}\left|\frac{\partial^{2}F}{\partial x_{d+1}^2}(x)\right|\] and
	\[\|F\|_{\min,2}\coloneqq \min_{x\in [0,1]^{d+1}}\left|\frac{\partial^{2}F}{\partial x_{d+1}^2}(x)\right|.\]
	Our main result for non-linear pushforwards of fractal measures is Theorem~\ref{t:main}. It is proved in Section~\ref{s:mainthm}, which constitutes much of the work in this paper. 
	\begin{thm}
		\label{t:main}
		Let $\Phi$ be a fibre product CIFS for some $\Psi=\{\psi_j\}_{j\in J}$ and $\{\Psi_{j}\}_{j\in J}$. 
		Let $\p=(p_{j,l})_{j\in J,l\in L_j}$ be a probability vector and assume that there exists $\tau>0$ such that 
		\begin{equation}\label{e:largedevassump} 
			\sum_{j \in J} \sum_{l \in L_j} p_{j,l} |r_{j,l}|^{-\tau} < \infty.
		\end{equation}
		Let $\mu$ be the stationary measure for $\Phi$ and $\p$. Then there exist $\eta,\kappa,C>0$ such that the following holds. 
		For all $C^{2}$ functions $F\colon [0,1]^{d+1}\to\mathbb{R}$ which satisfy $\frac{\partial^{2}F}{\partial x_{d+1}^2}(x) \neq 0$ for all $x\in[0,1]^{d+1}$, we have 
		\[ |\widehat{F\mu}(\xi)|\leq C (1+\|F\|_{\infty,1}+\|F\|_{\infty,1}^{-\kappa}+\|F\|_{\infty,2}) (1+\|F\|_{\min,2}^{-\kappa})|\xi|^{-\eta} \quad \mbox{ for all } \xi \neq 0. \] 
		
		In particular, if we only assume that $F\colon [0,1]^{d+1}\to\mathbb{R}$ is a $C^2$ function which satisfies $\frac{\partial^{2}F}{\partial x_{d+1}^2}(x) \neq 0$ for all $x\in \supp(\mu)$, then there exists $C_{F,\mu}>0$ depending upon $F$ and $\mu$ such that 
		\[|\widehat{F\mu}(\xi)|\leq C_{F,\mu}|\xi|^{-\eta}\quad \mbox{ for all } \xi \neq 0.\] 
	\end{thm}
	
	Theorem~\ref{t:main} immediately implies the following statement for stationary measures for a CIFS consisting of similarities acting on $[0,1]$. 
	
	\begin{cor}\label{t:self-similar}
		Let $\Phi$ be a non-trivial CIFS acting on $[0,1]$ consisting of similarities, and let $\p$ be a probability vector. Assume that 
		\[ \sum_{a\in \A}p_a|r_a|^{-\tau}<\infty\]
		 for some $\tau>0$. Let $\mu$ be the stationary measure for $\Phi$ and $\p$. 
		Then there exist $\eta,\kappa,C>0$ such that the following holds. 
		For all $C^2$ functions $F\colon [0,1]\to \mathbb{R}$ which satisfy $F''(x) \neq 0$ for all $x \in [0,1]$, we have  
		\begin{align*} |\widehat{F\mu}(\xi)|\leq C&\Big(1+\max_{x\in [0,1]}|F'(x)|+\Big(\max_{x\in [0,1]}|F'(x)|\Big)^{-\kappa}+\max_{x\in [0,1]}|F''(x)|\Big)\\*
			& \times \Big(1+\Big(\min_{x\in [0,1]}|F''(x)|\Big)^{-\kappa}\Big)|\xi|^{-\eta}
			\end{align*}
		for all $\xi\neq 0$.
		
		In particular, if we only assume that $F\colon [0,1]\to\mathbb{R}$ is a $C^{2}$ function which satisfies $F''(x) \neq 0$ for all $x\in \supp(\mu)$, then there exists $C_{F,\mu}>0$ such that 
		 \[|\widehat{F\mu}(\xi)|\leq C_{F,\mu}|\xi|^{-\eta}\quad \mbox{ for all } \xi \neq 0. \]
	\end{cor} 
	\begin{proof}
		Let $\Phi=\{\varphi_a\}_{a\in \A}$ and $\p$ respectively be a CIFS and probability vector satisfying our assumptions, and let $\mu$ be the corresponding stationary measure. Let $\Psi=\{\psi_{0}\}$ be the single-element self-similar IFS acting on $\R$ consisting of the map $\psi_{0}(x)=x/2$. Let $\Psi_{0}=\{\varphi_a\}_{a\in \A}$ and $\tilde{\Phi}$ be the corresponding fibre product IFS. Let $\nu$ be the stationary measure (on $\R^2$) corresponding to $\tilde{\Phi}$ and $\p$. It is straightforward to verify that $\nu = \delta_{0}\times \mu$, where $\delta_0$ is the Dirac mass at $0$. 
		
		Now let $F \colon [0,1]\to \R$ be a $C^2$ function satisfying $F''(x)\neq 0$ for all $x\in[0,1]$ (or all $x \in \supp(\mu)$). We define $\tilde{F} \colon [0,1]^{2}\to \R$ by $\tilde{F}(x,y)=F(y)$. Then $\tilde{F}\nu=F\mu$. Now applying Theorem~\ref{t:main} for $\tilde{F}$ and $\nu$ implies the desired statements for $F$ and $\mu$.
	\end{proof}

	We now make several comments about the statement of Theorem~\ref{t:main} and Corollary~\ref{t:self-similar}. 
The first part of Corollary~\ref{t:self-similar} will be used in an important way in our proof of Theorem \ref{thm:Analytic pushforward thm}. The bound involving the $\max_{x\in [0,1]}|F'(x)|,\max_{x\in [0,1]}|F''(x)|$ and $\min_{x\in [0,1]}|F'(x)|$ terms will be sufficiently flexible to allow us to control pushforwards where the underlying $F$ may satisfy $F''(x)=0$ for some $x\in \supp(\mu)$. 
The second part of Corollary~\ref{t:self-similar} in the special case of self-similar measures has been obtained independently and simultaneously by Algom et al. in~\cite{AlgomPreprintnonlinear}, using a somewhat different method. 
	As well as holding for IFSs, Theorem~\ref{t:main} and Corollary~\ref{t:self-similar} hold for CIFSs, where the condition~\eqref{e:largedevassump} is important as it allows us to apply Cram\'er's theorem for large deviations. This condition also implies that the \emph{Lyapunov exponent} $\Lambda$ is finite, where we define
	\begin{equation}\label{e:lyapunov} 
    \Lambda \coloneqq \sum_{j \in J} \sum_{l \in L_j} p_{l,j} \log |r_{l,j}^{-1}|. \end{equation}
	We emphasise that we do not assume that the CIFS satisfies any homogeneity or separation assumptions (a well-studied example of a CIFS of similarities which does satisfy nice properties is the L\"uroth maps). We also emphasise that the base CIFS $\Psi$ can be an arbitrary CIFS (the contractions do not need to be ``nice'' maps). The constants $\eta,\kappa,C>0$ do not depend upon the choice of $F$ and only depend upon the underlying CIFS and probability vector. 
	Given a particular CIFS, one could carefully follow the proof to make the constant $\eta$ explicit. However, this argument does not yield a general formula for $\eta$, and as such we will not pursue this here. 
	
\begin{figure}[ht]
	\subfloat[The attractor for $\Phi_1$.]{
		\begin{minipage}[c][1\width]{
				0.3\textwidth}
			\centering
			\includegraphics[width=1\textwidth]{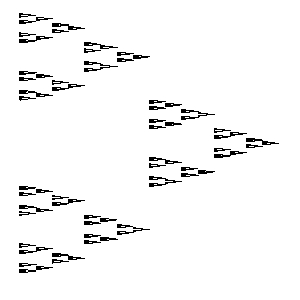}
	\end{minipage}}
	\hfill 	
	\subfloat[The attractor for $\Phi_2$.]{
		\begin{minipage}[c][1\width]{
				0.3\textwidth}
			\centering
			\includegraphics[width=1\textwidth]{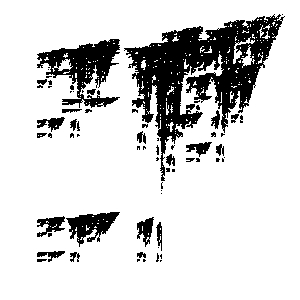}
	\end{minipage}}
	\hfill	
	\subfloat[The attractor for $\Phi_3$.]{
		\begin{minipage}[c][1\width]{
				0.3\textwidth}
			\centering
			\includegraphics[width=1\textwidth]{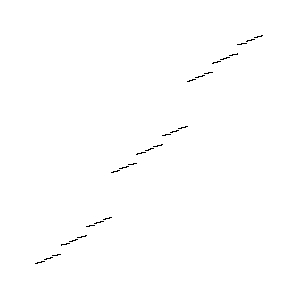}
	\end{minipage}}
	\caption[]{Let \\
		\begin{minipage}{\linewidth}
			\begin{align*} \Phi_{1}=\Big\{&\varphi_{1}(x,y)=\Big(\frac{x}{2},\frac{y}{3}\Big),\, \varphi_{2}(x,y)=\Big(\frac{x}{2},\frac{y+2}{3}\Big),\\ & \varphi_{3}(x,y)=\Big(\frac{x+1}{2},\frac{y+1}{3}\Big)\Big\}, \\
				\Phi_{2}=\Big\{&\varphi_{1}(x,y)=\Big(\frac{x}{3},\frac{y}{5}\Big),\, \varphi_{2}(x,y)=\Big(\frac{x}{3},\frac{4y+5}{10}\Big),\\
				&\varphi_{3}(x,y)=\Big(\frac{x+1}{2},\frac{y+1}{2}\Big),\,	\varphi_{4}(x,y)=\Big(\frac{x+2}{5},\frac{9y}{10}\Big),\\ &\varphi_{5}(x,y)=\Big(\frac{7x+1}{10},\frac{3y+6}{10}\Big),\,
				\varphi_{6}(x,y)=\Big(\frac{3x+6}{10},\frac{2x+2}{5}\Big)\Big\}, \\
				\Phi_{3}=\Big\{&\varphi_{1}(x,y)=\Big(\frac{x}{3},\frac{y}{5}\Big),\, \varphi_{2}(x,y)=\Big(\frac{x+1}{3},\frac{y+2}{5}\Big),\\
				&\varphi_{3}(x,y)=\Big(\frac{x+2}{3},\frac{y+4}{5}\Big)\Big\}.
			\end{align*}
		\end{minipage} 
	\vspace{1mm}
	
		The attractor for $\Phi_{1}$ is a Bedford--McMullen carpet to which Theorem~\ref{t:main} can be applied, and $\Phi_{2}$ is an overlapping IFS to which Theorem~\ref{t:main} applies. Notice that for both $\Phi_{1}$ and $\Phi_{2}$ the maps $\varphi_{1}$ and $\varphi_{2}$ have the same horizontal component but have vertical components with distinct fixed points. Therefore $\Phi_{1}$ and $\Phi_{2}$ can both be realised as fibre product CIFSs. Notice however that $\Phi_{3}$ cannot be realised as a fibre product CIFS because the vertical slices through this set always consist of singletons.}\label{f:ifs}
\end{figure}
	
	\subsection{Structure of the paper} 
	In Section~\ref{s:applications} we prove several applications, assuming Theorems~\ref{thm:analyticthm} and~\ref{t:main}. These applications relate to the Fourier uniqueness problem, Fractal Uncertainty Principles, Fourier restriction estimates, quantitative equidistribution properties of numbers in fractal sets, and conditional mixing. 
	In Section~\ref{s:analyticthm} we prove Theorem~\ref{thm:Analytic pushforward thm} and observe that combining this with existing results in the literature \cite{AHW,BS} gives Theorem~\ref{thm:analyticthm}. 
	In Section~\ref{s:mainthm} we prove Theorem~\ref{t:main}; an informal outline of the proof, which involves disintegrating the measure~$\mu$, using large deviations theory, and applying an Erd\H{o}s--Kahane-type argument, is given in Section~\ref{s:outline}. 
	We also prove that the Fourier transform of countably generated self-similar measures decays outside a sparse set of frequencies (see Corollary~\ref{c:tsujii}), which may be of interest in its own right. 
	We will conclude in Section~\ref{Future directions} with some discussion of future directions. 
	
\subsection{Notation}
Throughout this paper we will adopt the following notational conventions. We write $\O(X)$ to denote a quantity bounded in modulus by $CX$ for some $C>0$. We also write $X \preceq Y$ if $X =\O(Y)$, and $X \approx Y$ if $X \preceq Y$ and $Y \preceq X$. We write $\O_{k}(X)$, $X\preceq_{k} Y$, or $X\approx_{k} Y$ when we want to emphasise that the underling constant $C$ depends upon some parameter $k$. We similarly write $o_k(1)$ to denote a function of $k$ which tends to $0$ as $k \to \infty$. 
We let $e(y)=e^{-2\pi i y}$ for $y\in \R$.

We will also use the following standard notation from fractal geometry. Given a CIFS $\Phi=\{\varphi_{a}\}_{a\in \A}$ with digit set $\A$, we let $\A^*=\cup_{n=1}^{\infty}\A^n$. Moreover, given $\a=(a_1,\ldots,a_n)\in \A^*$ we let $\varphi_{\a}=\varphi_{a_1}\circ \cdots \circ \varphi_{a_n}$. If a CIFS consists of similarities then for each $\a=(a_1,\ldots,a_n)\in \A^{*}$ we let $r_{\a}=\prod_{j=1}^{n}r_{a_j}$ denote the product of the contraction ratios. Similarly, given a probability vector $\p=(p_a)_{a\in \A}$ and a word $\a=(a_1,\ldots,a_n)\in \A^{*}$ we let $p_{\a}=\prod_{j=1}^{n}p_{a_j}$.

	\section{Applications}\label{s:applications}
	Equipped with suitable knowledge about the Fourier transform of a measure one can derive a number of interesting applications. In this section we will detail several consequences of our results.
	
	\subsection{The Fourier uniqueness problem} 
	A set $X\subset[0,1]$ is called a set of uniqueness if every trigonometric series \[\sum_{n\in \mathbb{Z}}a_ne^{2\pi ix}\] with coefficients $a_n\in \mathbb{C}$ that takes the value $0$ for all $x \in [0,1]\setminus X$ is a trivial trigonometric series in the sense that $a_n=0$ for all $n\in \mathbb{Z}$. 
	It is known that every countable closed set is a set of uniqueness, whereas every set of positive Lebesgue measure is a set of multiplicity. 
	We refer the reader to \cite{KL} and \cite{JialunSahlsten1} for a more detailed introduction to this topic. If a set $X$ supports a Rajchman measure, then a result of Salem \cite{Salem} asserts that $X$ is a set of multiplicity for the Fourier uniqueness problem. Combining this result with Theorem~\ref{thm:analyticthm} immediately implies the following statement.
	
	\begin{thm}
		Let $X$ be the attractor for a IFS $\Phi=\{\varphi_a \colon [0,1] \to [0,1] \}_{a\in \A}$ consisting of analytic maps. If there exists $a \in \A$ such that $\varphi_{a}$ is not an affine map then $X$ is a set of multiplicity. 
	\end{thm}
	Theorem~\ref{t:main} implies an analogous statement about non-linear images of fractal sets being sets of multiplicity. 
	
	\subsection{Fractal Uncertainty Principles}
Fractal Uncertainty Principles are tools that (roughly) state that a function cannot be localised in position and frequency near a fractal set. 
This idea is made precise as follows: we say that sets $X,Y\subset\mathbb{R}^d$ satisfy a Fractal Uncertainty Principle at the scale $h>0$ with exponent $\beta>0$ and constant $C>0$ if for all $f\in L^{2}(\R^d)$ with 
\[\{ \xi \in \R^d : \widehat{f}(\xi) \neq 0\} \subset \{h^{-1}y:y\in Y\}, \] 
we have 
\[\|f\|_{L^2(X)} \leq Ch^{\beta} \|f\|_{L^2(\R^d)}.\] 
Here, $L^2(X)$ and $L^2(\R^d)$ are defined using Lebesgue measure. 
The existence of Fractal Uncertainty Principles for neighbourhoods of fractal sets arising from hyperbolic dynamics has had significant applications in the field of quantum chaos (see \cite{DJ,DJN}). For more on Fractal Uncertainty Principles we refer the reader to the survey by Dyatlov \cite{Dyatlov} and the paper of Dyatlov and Zahl \cite{DZ}. 

Given a finite Borel measure $\mu$ supported on $\mathbb{R}$, let 
\begin{align*}
\delta(\mu) \coloneqq \inf\{\delta\geq 0:\exists& C>0 \textrm{ s.t. } \forall x\in \supp(\mu) \textrm{ and } r\in (0,1)\\*
&\textrm{ we have } \mu(B(x,r))\geq Cr^{\delta}\}.
\end{align*}The quantity $\delta(\mu)$ is the upper Minkowski dimension of $\mu$ introduced in \cite{FFK}. Moreover, \cite[Theorem~2.1]{FFK} asserts that for all compact $K \subset \R^d$, 
\[ \overline{\dim}_{\mathrm B} K = \min\{ \delta(\mu) : \mu \textrm{ is a fully supported finite Borel measure on } K \},\] 
so in the theorem below $\delta_j \geq \dim_{\mathrm B} K_j = \dim_{\mathrm H} K_j$. 

Using \cite[Proposition~1.5]{BS} together with the Fourier decay guaranteed by Theorem~\ref{thm:analyticthm}, we immediately have the following statement.

\begin{thm}
	\label{t:FUP}
Let $K_1,K_2\subset \R$ be the attractors for two analytic IFSs $\Phi_{1}$ and $\Phi_{2}$, both acting on $[0,1]$. 
Assume that $\Phi_{2}$ contains a contraction that is not an affine map. 
For $j\in \{1,2\}$, 
let \[\delta_{j}=\inf\{\delta(\mu): \mu \textrm{ is a self-conformal measure for }\Phi_{j}\},\]
and assume that $\delta_1 + \delta_2 < 1$. 
Then for all $\kappa \in (\delta_{1}+\delta_{2},1)$ there exists $C>0$ (depending upon $\kappa$, $\Phi_1$, $\Phi_2$) such that for all $h>0$, the Fractal Uncertainty Principle is satisfied for $X=K_1+B(0,h)$ and $Y=K_{2}+B(0,h)$ at scale $h$ with exponent $\beta = \frac{1}{2} - \frac{\kappa}{2}$ and constant $C$.
\end{thm}

By considering suitable iterates of the $\Phi_{j}$, it is possible to improve the exponent $\beta$. Indeed, if $\Phi_{1}$ and $\Phi_{2}$ satisfy the strong separation condition then for any $\kappa>\dim_{\mathrm H} K_{1}+\dim_{\mathrm H} K_{2}$ the conclusion of Theorem~\ref{t:FUP} holds. 
We expect that a similar statement holds without assuming the strong separation condition. 
Unlike the Fractal Uncertainty Principle of Bourgain and Dyatlov \cite[Theorem~4]{BD2} we make no assumption of porosity or Ahlfors--David regularity. On the other hand, Theorem~\ref{t:FUP} can be applied only if $\dim_{\mathrm H} K_{1}+\dim_{\mathrm H} K_{2} < 1$, and we make a non-linearity assumption on $\Phi_2$ to guarantee Fourier decay. 

\subsection{Fourier restriction estimates}
If $1<p \leq 2$, the classical Hausdorff--Young inequality allows one to view the Fourier transform as a map $L^p(\R^d) \to L^q(\R^d)$, where $1/p + 1/q = 1$. Therefore at first sight $\hat{f}$ may appear to be definable only almost everywhere. 
Fourier restriction theory relates to the interesting fact that $\hat{f}$ can in fact be restricted to certain sets of zero Lebesgue measure in a meaningful way, and our next application of Theorem~\ref{thm:analyticthm} is an example of this. 
	
\begin{thm}\label{t:restriction}
Let $\{\varphi_a\colon [0,1] \to [0,1]\}_{a\in \A}$ be an IFS such that each $\varphi_{a}$ is analytic, and suppose that there exists $a \in \mathcal{A}$ such that $\varphi_{a}$ is not an affine map. 
Let $X$ be the attractor and $\mu$ be any self-conformal measure for the IFS (so $\supp (\mu) = X$). 
Then there exists $p_{\mu}>1$ such that for all $p \in [1,p_{\mu}]$ there exists $C_p >0$ such that 
\[ \left( \int_X |\hat{f}(\xi)|^2 d\mu(\xi)\right)^{1/2} \leq C_p ||f||_{L^p(\R)} \]
for all Schwartz functions $f \colon \R \to \mathbb{C}$. 
Therefore the linear operator $f \mapsto \hat{f}$ extends uniquely to a bounded linear operator $L^p(\R,\mbox{Lebesgue}) \to L^2(X,\mu)$. 
\end{thm}
\begin{proof}
By Theorem~\ref{thm:analyticthm}, $\mu$ has polynomial Fourier decay. 
It is straightforward to deduce (see the short argument from \cite[page~5]{Mitsis2002restriction} for example) that $\mu$ has a positive Frostman exponent, meaning that there exists $C,s>0$ such that $\mu(B(x,r)) \leq C r^s$ for all $x \in \R$. 
We note that the existence of such a Frostman exponent could also be proved more directly by adapting the proof of \cite[Proposition 2.2]{FL}. 
 Given these two properties of $\mu$, the restriction estimate follows from a result of Mockenhaupt \cite[Theorem~4.1]{Mockenhaupt2000restriction} (see also Mitsis \cite[Corollary~3.1]{Mitsis2002restriction} and Stein's earlier argument \cite[page~353]{Stein1993harmonic}). 
\end{proof}

	\subsection{Normal numbers and effective equidistribution}
	Given an integer $b\geq 2$, a real number $x$ is said to be normal in base $b$ if the sequence $(b^nx)_{n=1}^{\infty}$ is uniformly distributed modulo one. This property is equivalent to the base-$b$ expansion of $x$ observing all finite words with the expected frequency. A well-known result due to Borel~\cite{Bor} states that Lebesgue almost every $x$ is normal in every base. 
	Despite this result, it is a notoriously difficult problem to verify that a given real number (such as $\pi$, $e$ or $\sqrt{2}$) is normal in a given base. A more tractable line of research is the problem of determining when a Borel probability measure gives full mass to the set of normal numbers in a given base. This problem was considered by Cassels in \cite{Cas}, who proved that with respect to the natural measure on the middle third Cantor set, almost every $x$ is normal in base $b$ if $b$ is not a power of $3$. In~\cite{Hos} Host proved that if $p,q\in\mathbb{N}$ satisfy $\gcd(p,q)=1$ and $\mu$ is an invariant and ergodic measure for the map $x \mapsto px \mod 1$ with positive entropy, then $\mu$-almost every $x$ is normal in base $q$. This result was generalised by Lindenstrauss~\cite{Lindenstrauss}, and then generalised further by Hochman and Shmerkin in~\cite{HocShm} to the case where $\frac{\log p}{\log q}\notin\mathbb{Q}$. Hochman and Shmerkin also proved that for a stationary measure $\mu$ for a well separated self-similar IFS, $\mu$-almost every $x$ is normal in base $b$ for every self-similar measure if $\frac{\log |r_a|}{\log b}\notin\mathbb{Q}$ for some contraction ratio $r_a$. 
	 Hochman and Shmerkin's result was extended by Algom, the first named author, and Shmerkin in \cite{ABS} to the case where the IFS could be potentially overlapping. Recently B\'{a}r\'{a}ny et al. \cite{BKPW} further generalised this result to apply to self-conformal measures. We also refer the reader to the work of Dayan, Ganguly and Weiss~\cite{DGW}, who considered a complementary approach. 
	 They showed that under suitable irrationality assumptions for the translations appearing in the IFS one can prove normality results. For more on this topic we refer the reader to the book by Bugeaud \cite{Bug} and the references therein. 
     
     In this paper we will show that if $\Phi$ is an IFS consisting of analytic maps, one of which is not an affine map, then $\mu$-almost every $x$ is normal in all bases for every self-conformal measure $\mu$. In fact we are able to show that for $\mu$-almost every $x$ a more quantitative normality result holds (see Theorem~\ref{t:Normal thm} below). 
     These results will be proved by using the fact that our measure has polynomial Fourier decay and therefore existing results in the literature can be applied. 
     In particular, a classical criterion of Davenport, Erd\H{o}s, and LeVeque \cite{DEL} gives normality of $\mu$-typical points in all bases, and more recent results of Pollington et al.~\cite{PVZZ} immediately give \emph{effective} equidistribution results for $q_n x \mod 1$ for certain strictly increasing sequences $(q_n)$, which can be improved further in the case $q_n = b^n$ for some natural number $b$. 
	 
	 In addition to applications on normal numbers in fractal sets, our main results will also be applied to certain shrinking target problems and Diophantine approximation with $b$-adic rationals. Given a sequence of real numbers $(q_n)_{n=1}^{\infty}$, a function $\psi\colon \N\to[0,1/2]$, and $\gamma\in [0,1]$, we associate the set \[W((q_n),\gamma,\psi)\coloneqq \{x\in [0,1]:\|q_{n}x-\gamma\|\leq \psi(n)\textrm{ for infinitely many }n\in \N\}.\]  Here and throughout we let $\|\cdot\|$ denote the distance to the integers, i.e. $\|x\|=d(x,\mathbb{Z})$ for all $x\in \mathbb{R}$. 
	 For some of our results we will assume that the sequence $(q_{n})_{n=1}^{\infty}$ is lacunary, which means that there exists $K>1$ such that $q_{n+1}/q_{n}>K$ for all $n\in \N$.
	 
	 In the special case where $(q_n)_{n=1}^{\infty}=(b^n)_{n=1}^{\infty}$ we will denote 
	 $W((q_n),\gamma,\psi)$ by $W_{b}(\gamma,\psi)$. Taking $\gamma=0$ we see that $W_{b}(0,\psi)$ coincides with the set of $x$ that can be approximated by a $b$-adic rational $\frac{p}{b^{n}}$ with error $\frac{\psi(n)}{b^n}$ for infinitely many $n$. In \cite{Phi} Philipp proved that for all $b\geq 2$, $\psi\colon \N\to[0,1/2]$, and $\gamma\in [0,1]$ we have 
	 \[ \L(W_{b}(\gamma,\psi)) = \left\{ \begin{array}{ll}
	 	1 & \mbox{if $\sum_{n=1}^{\infty}\psi(n)=\infty$};\\
	 	0 & \mbox{if $\sum_{n=1}^{\infty}\psi(n)<\infty$}.\end{array} \right. \] 
	 	Here $\L$ denotes the one-dimensional Lebesgue measure. The convergence part of Philipp's result is simply the Borel--Cantelli lemma. The more interesting and technically demanding statement is the divergence part. As is the case for normal numbers, it is difficult to verify that specific constants belong or do not belong to a given $W_{b}(\gamma,\psi)$. 
	 	Therefore it is natural to consider the related problem of determining whether a probability measure exhibits analogous behaviour to that observed for the Lebesgue measure. This idea is expressed more formally in the following conjecture due to Velani\footnote{Velani originally formulated this conjecture with $\psi$ monotonically decreasing and $\gamma=0$, as it is written in~\cite{ACY}.}.
 	\begin{conjecture}
 		\label{Velani conjecture}
 		Let $\psi\colon \mathbb{N}\to [0,1/2]$, $\gamma\in[0,1]$, and let $\mu$ be the natural measure on the middle third Cantor set. Then
 		\[ \mu(W_{2}(\gamma,\psi))=
 		\begin{cases}
 			1 \quad \text{ if }\sum_{n=1}^{\infty}\psi(n)=\infty, \\
 			0 \quad \text{ if }\sum_{n=1}^{\infty}\psi(n)<\infty.
 		\end{cases}
 		\]
 	\end{conjecture} 
	 Initial progress towards Conjecture~\ref{Velani conjecture} has been made in \cite{ABCY,ACY,Bak}. Velani's conjecture is formulated for the natural measure on the middle third Cantor set, but one could just as easily ask whether the same conclusion holds for any fractal measure. In Theorem~\ref{t:Normal thm} below we will prove the following statement which shows that a version of Velani's conjecture holds when $\mu$ is a self-conformal measure coming from a suitable IFS. 
     In fact, in this theorem and others, we will prove that a wide class of fractal measures satisfy the following three equidistribution properties. 
     In what follows, given $\psi\colon \N\to [0,1/2]$ we let 
	 \[ \Sigma(N) \coloneqq \sum_{n=1}^{N}\psi(n).\] 
     \begin{itemize}
        \item \textbf{Property~(A):} We say that a measure $\mu$ satisfies this property if given any sequence of positive real numbers $(q_n)_{n=1}^{\infty}$ satisfying $\inf_{n\in \N} (q_{n+1}-q_{n})>0$, the sequence $(q_n x)_{n=1}^{\infty}$ is uniformly distributed modulo one for $\mu$-almost every $x$. 
        \item \textbf{Property~(B):} We say that a measure $\mu$ satisfies this property if given any lacunary sequence $(q_n')_{n=1}^{\infty}$ of natural numbers, $\psi\colon \mathbb{N}\to [0,1/2]$, $\gamma\in [0,1]$ and $\epsilon>0$, then 
	    \begin{equation}\label{e:propb}
        \#\{1\leq n\leq N:\|q_n' x-\gamma\|\leq \psi(n)\}=2\Sigma(N)+\O\left(\Sigma(N)^{2/3}\left(\log (\Sigma(N)+2)\right)^{2+\epsilon}\right)
        \end{equation}
	    for $\mu$-almost every $x$.
        \item \textbf{Property~(C):} We say that a measure $\mu$ satisfies this property if for all integers $b\geq 2$, $\psi\colon \mathbb{N}\to [0,1/2]$, $\gamma\in [0,1]$ and $\epsilon>0$, one has   
	    \begin{equation}\label{e:propc}
        \#\{1\leq n\leq N:\|b^{n} x-\gamma\|\leq \psi(n)\}=2\Sigma(N)+\O\left(\Sigma(N)^{1/2}\left(\log (\Sigma(N)+2)\right)^{2+\epsilon}\right)
        \end{equation}
	    for $\mu$-almost every $x$. 
     \end{itemize}

     \begin{prop}[\cite{DEL,PVZZ}]\label{p:equi}
         If a measure $\mu$ on $\R$ has polynomial Fourier decay then it satisfies Properties~(A), (B) and (C). 
     \end{prop}
     \begin{proof}
     Property~(A) is a consequence of the well-known criterion of Davenport, Erd\H{o}s, and LeVeque~\cite{DEL}. 
     Property~(B) holds by \cite[Theorem~1]{PVZZ} and Property~(C) holds by \cite[Theorem~3]{PVZZ} due to Pollington et al. 
     \end{proof}

       We make two remarks about Property~(C). 
    \begin{remark}
        \begin{itemize}
            \item For measures with fast enough (for example polynomial) Fourier decay, Property~(C) also holds for certain non-lacunary sequences, for example when $b^n$ is replaced by the increasing sequence of numbers of the form $\{2^n 3^m\}$ for $n,m \in \N$. We refer the reader to \cite[Theorem~3]{PVZZ} for the most general formulation. 
            \item Ignoring the logarithmic error terms, for the class of self-conformal measures we consider one cannot hope to improve the exponent $1/2$ in Property~(C). This is a consequence of the law of the iterated logarithm. 
        \end{itemize}
    \end{remark}
	
	\begin{thm}
	\label{t:Normal thm}
	Let $\Phi=\{\varphi_a \colon [0,1] \to [0,1] \}_{a\in \A}$ be an IFS such that each $\varphi_{a}$ is analytic. Suppose that there exists $\varphi_{a}$ which is not an affine map. Then every self-conformal measure $\mu$ satisfies Properties~(A), (B) and (C). 
	\end{thm}
	\begin{proof}
        By Theorem~\ref{thm:analyticthm}, $\mu$ has polynomial Fourier decay. The result therefore follows from Proposition~\ref{p:equi}. 
	\end{proof}
  
	We also have the following result for non-linear pushforwards. 
    We emphasise that a pushforward measure $F\mu$ on $\R$ satisfies Properties~(A), (B) and (C) if and only if for all $(q_n)$, $(q_n')$, $\psi$, $\gamma$ and $\varepsilon$ as in the statements of these properties, for $\mu$-almost every $y$, the real number $x \coloneqq F(y)$ satisfies~\eqref{e:propb} and~\eqref{e:propc} and $(q_n x)_{n=1}^{\infty}$ is equidistributed modulo~$1$. 
	
	\begin{thm}
		\label{t:normalpushforward}
		Let $\Phi$ be a fibre product CIFS for some $\Psi=\{\psi_j\}_{j\in J}$ and $\{\Psi_{j}\}_{j\in J}$. 
		Let $\p=(p_{j,l})_{j\in J,l\in L_j}$ be a probability vector and assume that there exists $\tau>0$ such that 
		\begin{equation*}
			\sum_{j \in J} \sum_{l \in L_j} p_{j,l} |r_{j,l}|^{-\tau} < \infty.
		\end{equation*}
		Let $\mu$ be the stationary measure for $\Phi$ and $\p$. Suppose that $F\colon [0,1]^{d+1}\to \R$ is a $C^2$ function such that $\frac{\partial^{2}F}{\partial x_{d+1}^2}(x) \neq 0$ for $\mu$-almost every $x$. 
        Then the pushforward measure $F\mu$ satisfies Properties~(A), (B) and (C). 
	\end{thm}

\begin{proof}
Let $\Phi=\{\varphi_{j,l}\}_{j\in J,l\in L_j}$, $\p$, and $F$ satisfy the assumptions of our theorem. Let $\A=\{(j,l):j\in J,\, l\in L_{j}\}$, and let $\pi\colon \A^{\N}\to [0,1]^{d+1}$ be the map given by \[ \pi((a_i))=\lim_{n\to\infty}(\varphi_{a_1}\circ \cdots \circ \varphi_{a_n})(0).\] 
We also let $m$ denote the infinite product measure on $\A^{\N}$ corresponding to the probability vector $\p$. Importantly, the stationary measure $\mu$ satisfies $\mu=\pi m$. Using this property and our assumptions on $F$, we have that $\frac{\partial^{2}F}{\partial x_{d+1}^2}(\pi((a_i))) \neq 0$ for $m$-almost every $(a_i)$. 
Appealing now to the fact $F$ is a $C^2$ function, we see that for $m$-almost every $(a_i)$ there exists a prefix $(a_1,\ldots,a_n)$ of minimal length such that $\frac{\partial^{2}F}{\partial x_{d+1}^2}(x) \neq 0$ for all $x\in (\varphi_{a_1}\circ \cdots \circ \varphi_{a_n})([0,1]^{d+1})$. 
We let $G_{F}\subset \A^{*}$ denote the countable set of prefixes arising from this observation. The set $G_{F}$ satisfies the following properties:
\begin{enumerate}
	\item\label{i:nonprefix} If $\a,\b\in G_{F}$ are distinct then $\a$ is not a prefix of $\b$, and $\b$ is not a prefix of $\a$. 
	\item\label{i:fullmeas} \[ m\left(\bigcup_{\a\in G_{F}}\{(b_i)\in \A^{\N}:\a \text{ is a prefix of }(b_i)\}\right)=1.\]
\end{enumerate}
Property~\eqref{i:nonprefix} follows because by definition $G_{F}$ is the set of prefixes of minimal length satisfying $\frac{\partial^{2}F}{\partial x_{d+1}^2}(x) \neq 0$ for all $x\in (\varphi_{a_1}\circ \cdots\circ \varphi_{a_n})([0,1]^{d+1})$. 
Property~\eqref{i:fullmeas} holds because $m$-almost every $(a_i)$ is contained in one of the sets appearing in this union. Using properties~\eqref{i:nonprefix} and~\eqref{i:fullmeas} and iterating the self-similarity relation $\mu=\sum_{a\in \A}p_{a}\varphi_{a}\mu$ gives  
\[\mu=\sum_{\a\in G_{F}}p_{\a}\varphi_{\a}\mu.\] 
This in turn implies that 
\begin{equation}
	\label{e:pushforward split}
F\mu=\sum_{\a\in G_{F}}p_{\a}(F\circ \varphi_{\a})\mu.
\end{equation} 

Appealing to the definition of $G_{F}$ and the chain rule, for each $\a\in G_{F}$, the map $F\circ \varphi_{\a}$ satisfies 
\[ \frac{\partial^{2}(F\circ \varphi_{\a})}{\partial x_{d+1}^{2}}\neq 0 \] 
for all $x\in [0,1]^{d+1}$. By Theorem~\ref{t:main}, this implies that for all $\a\in G_{F}$ there exists $\eta>0$ such that $|\widehat{(F\circ \varphi_{\a})\mu}(\xi)|\preceq_{F,\a} |\xi|^{-\eta}$ for all $\xi\neq 0$. By Proposition~\ref{p:equi}, $(F\circ \varphi_{\a})\mu$ satisfies Properties~(A), (B) and~(C) for any $\a\in G_{F}$. Combining this observation with \eqref{e:pushforward split} completes our proof. 
\end{proof}
Note that in the statement of Theorem~\ref{t:normalpushforward} we only assume $\frac{\partial^{2}F}{\partial x_{d+1}^2}(x) \neq 0$ for $\mu$-almost every $x$, which is why the result was not just an immediate consequence of Theorem~\ref{t:main} and Proposition~\ref{p:equi}. This weaker assumption allows for greater flexibility and yields Corollaries~\ref{c:selfaffine} and~\ref{c:analyticnormal}. 
The IFS considered in Corollary~\ref{c:selfaffine} is of Bedford--McMullen type \cite{Bed,MC}. The arguments given in its proof apply more generally to Gatzouras--Lalley carpets~\cite{GL} and Bara\'nski carpets~\cite{Bar}, but for simplicity we content ourselves with this application to Bedford--McMullen carpets. The IFS $\Phi_{1}$ in Figure~\ref{f:ifs} gives an explicit example of a Bedford--McMullen carpet to which Corollary~\ref{c:selfaffine} applies. 

\begin{cor}\label{c:selfaffine}
Let $m,n\geq 2$ be integers. Let $\Phi=\{\varphi_{i,j}(x,y)=(\frac{x+i}{m},\frac{y+j}{n})\}_{(i,j)\in \A}$ for some $\A\subset \{0,\ldots,m-1\}\times \{0,\ldots,n-1\}$. Suppose that there exists $i\in \{0,\ldots,m-1\}$ such that $(i,j),(i,j')\in \A$ for distinct $j,j'\in \{0,\ldots,n-1\}$. Moreover, assume that there exists $(i_1,j_1),(i_2,j_2)\in \A$ such that $i_1\neq i_2$. 
Let $F\colon [0,1]^{2}\to \R$ be a multivariate polynomial of the form $F(x,y)=\sum_{(d_1,d_2)\in D}c_{d_1,d_2}x^{d_1}y^{d_2}$ for some finite set $D\subset \N^2$ where $c_{d_1,d_2}\neq 0$ for all $(d_1,d_2)\in D$. Assume that $d_2\geq 2$ for some $(d_1,d_2)\in D$. Let $\mu$ be a stationary measure for~$\Phi$ and some probability vector $\p$. 
Then the pushforward measure $F \mu$ satisfies Properties~(A), (B) and (C). 
\end{cor}

\begin{proof}
It is straightforward to see that this $\Phi$ can be realised as a fibre product CIFS. If we want to apply Theorem~\ref{t:normalpushforward} and complete our proof, it therefore remains to verify that $\frac{\partial^{2}F}{\partial y^2}(x,y) \neq 0$ for $\mu$-almost every $(x,y)$. By our assumption that $d_{2}\geq 2$ for some $(d_1,d_2)\in D$, it follows that $\frac{\partial^{2}F}{\partial y^2}(x,y)$ is a non-zero multivariate polynomial. To each $x\in [0,1]$ we associate the function $G_{x}(y)=\frac{\partial^{2}F}{\partial y^2}(x,y)$. We now make three simple observations:
\begin{enumerate}
	\item There exists a finite set $Z\subset [0,1]$ such that for all $x\in [0,1]\setminus Z$ the function $G_{x}$ is a non-zero polynomial, and therefore has finitely many roots. 
	\item The measure $\pi\mu$ is a non-atomic self-similar measure, where $\pi\colon [0,1]^{2}\to [0,1]$ is the projection map given by $\pi(x,y)=x$. This follows from our assumption that there exists $(i_1,j_1),(i_2,j_2)\in \A$ such that $i_1\neq i_2$.
	\item We can disintegrate $\mu$ along vertical fibres to obtain $\mu=\int \mu_{x}\, d(\pi\mu)(x)$. 
	Each $\mu_{x}$ is supported on the vertical fibre given by $\{(x',y)\in [0,1]^2:x'=x\}$. 
	From our assumption that there exists $i\in \{0,\ldots,m-1\}$ such that $(i,j),(i,j')\in \A$ for distinct $j,j'\in \{0,\ldots,n-1\}$, it is not hard to deduce that $\mu_{x}$ is non-atomic for $\pi\mu$-almost every $x$ (see Section~\ref{s:mainthm}). 
\end{enumerate}
Using these three observations, we have the following:
\begin{align*}
	&\mu\left(\left\{(x,y)\in [0,1]^2:\frac{\partial^{2}F}{\partial y^2}(x,y)=0\right\}\right)	\\
	&= \int_{[0,1]} \mu_{x}\left(\left\{(x,y)\in [0,1]^2:\frac{\partial^{2}F}{\partial y^2}(x,y)=0\right\}\right)\, d(\pi\mu)(x)\\
	&= \int_{[0,1]\setminus Z} \mu_{x}\left(\left\{(x,y)\in [0,1]^2:G_{x}(y)=0\right\}\right)\, d(\pi\mu)(x) \\*
	&= 0.
\end{align*}
Therefore all of the assumptions of Theorem~\ref{t:normalpushforward} are satisfied, and applying this theorem yields our result. 
\end{proof}

The following corollary extends \cite[Theorem~1.7]{HocShm}. It generalises this statement in three ways: it permits countably many contractions, the CIFS can contain overlaps, and it considers a more general family of $(q_n)_{n=1}^{\infty}$. No analogue of the counting parts of this corollary appeared in \cite{HocShm}. 
\begin{cor}\label{c:analyticnormal}
Let $\Phi$ be a non-trivial CIFS acting on $[0,1]$ consisting of similarities, and let $\p$ be a probability vector. Assume that 
\[ \sum_{a\in \A}p_a|r_a|^{-\tau}<\infty\]
for some $\tau>0$. Let $\mu$ be the stationary measure for $\Phi$ and $\p$. Let $F\colon [0,1]\to \mathbb{R}$ be an analytic function that is not an affine map. 
Then the pushforward measure $F \mu$ satisfies Properties~(A), (B) and (C). 
\end{cor}

\begin{proof}
Since $F$ is not an affine map, the function $F''$ is not the constant zero function. It is well known that the zeros of a non-zero analytic function cannot have an accumulation point, so $\#\{x:F''(x)=0\}<\infty$. 
Since $\mu$ is non-atomic, it follows that $\mu(\{x:F''(x)=0\})=0$.  
Using the same argument as in the proof of Corollary~\ref{t:self-similar}, our result now follows by applying Theorem~\ref{t:normalpushforward} to a stationary measure $\nu$ for a suitable choice of fibre product IFS and a suitable function $\tilde{F} \colon [0,1]^2 \to \R$. 
\end{proof}

\subsection{Conditional mixing}
Another application of Fourier decay results relates to mixing properties of chaotic dynamical systems. Indeed, Wormell~\cite{Wormell2022} has shown how information on the Fourier decay of stationary measures can be used to establish a phenomenon known as conditional mixing. 
She used results of Mosquera and Shmerkin~\cite{MS}, and Sahlsten and Stevens~\cite{SS}, to show that conditional mixing holds for a class of generalised baker’s maps. Using our results it is possible to deduce that conditional mixing holds for a wider class of such maps. In particular, in \cite[Theorem~4.2]{Wormell2022}, assumption~III can be relaxed in two ways: firstly, to remove the homogeneity assumption on the affine contractions, and secondly, to allow $\psi$ to be either $C^2$ with $\psi'' \neq 0$ (as in~III), or analytic and non-affine. 
	
	\section{Proof of Theorems \ref{thm:Analytic pushforward thm} and \ref{thm:analyticthm}}\label{s:analyticthm}
			In this section we prove Theorem~\ref{thm:Analytic pushforward thm} assuming Corollary~\ref{t:self-similar}. We will then combine Theorem~\ref{thm:Analytic pushforward thm} with results in~\cite{AHW,BS} to deduce Theorem~\ref{thm:analyticthm}. 

    \subsection{Proof of Theorem \ref{thm:Analytic pushforward thm}}
Before presenting our proof we need the following two lemmas. 
The first lemma is well known and says that our measure has a positive Frostman exponent; this property is often used when estimating the Fourier decay of a measure. 
	
	\begin{lemma}\cite[Proposition~2.2]{FL}.
		\label{l:FengFrostman}
		Let $\mu$ be a self-similar measure on $\R$. Then there exist $C,s>0$ such that \[\mu(B(x,r))\leq Cr^{s} \] for all $x\in\R$.
	\end{lemma}

	\begin{lemma}
		\label{l:Analytic level sets}
		Let $F\colon [0,1]\to \R$ be an analytic function which is not the constant zero function. Then there exist $C>0$, $k\in\mathbb{N}$ and $x_1,\ldots,x_n\in[0,1]$ such that for all $r$ sufficiently small, 
		\begin{equation}\label{e:analyticinclusion} 
		\{x\in [0,1]:|F(x)|<r\}\subseteq \bigcup_{i=1}^{n} B(x_i,Cr^{1/k}).
		\end{equation}
	\end{lemma}
\begin{proof}
Let $F\colon [0,1]\to \R$ be an analytic function which is not identically zero. We may assume that there exists $x\in [0,1]$ such that $F(x)=0$ (otherwise our result holds trivially). 
We let $x_1,\ldots,x_n$ denotes the zeros of $F$. Note that $F$ must have finitely many zeros because, as is well known, the zeros of a non-zero analytic function cannot have an accumulation point. 

 We now consider the power series expansion of $F$ around each $x_i$, that is the expansion \[ F(x)=\sum_{k=1}^{\infty}a_{k,i}(x-x_i)^{k}.\] Since $F$ is not the constant zero function, for each $i$ there exists a minimal $k_i\geq 1$ such that $a_{k_i,i}\neq 0$. This in turn implies that for each $x_{i}$, there exists $c_{i}> 0$ such that for all $x$ sufficiently close to $x_i$ we have 
 \[ c_{i}(x-x_{i})^{k_i}\leq |F(x)|.\] 
 Therefore for all $x$ sufficiently close to $x_i$, if $x$ satisfies $|F(x)|<r$ then we must have $x\in B\left(x_i,(r/c_i)^{1/k_i}\right)$. Taking $k=\max\{k_1,\dotsc,k_n\}$ and $C=\max\{c_1^{-1/k_1},\dotsc, c_{n}^{-1/k_n}\}$, it now follows from a compactness argument that for all $r$ sufficiently small the desired inclusion~\eqref{e:analyticinclusion} holds. 
 \end{proof}

We now prove Theorem~\ref{thm:Analytic pushforward thm}.
\begin{proof} 
First note that neither $F'$ nor $F''$ is the constant zero function. Applying Lemma~\ref{l:Analytic level sets} for both of these functions, we see that there exist $C>0$, $k\in \mathbb{N}$, and $x_{1},\ldots,x_n$ such that 
		\begin{equation}
			\label{e:levelsetinclusion}
		\{x\in [0,1]:|F'(x)|<r\}\cup	\{x\in [0,1]:|F''(x)|<r\}\subseteq \bigcup_{i=1}^{n} B(x_i,Cr^{1/k})
		\end{equation}
		 for all $r$ sufficiently small.

		Let 
		\[ \delta=\frac{\eta}{4(k+2)\kappa}, \] 
		where $k$ is as above and $\eta,\kappa>0$ are as in Corollary~\ref{t:self-similar} for the measure $\mu$. Let $\xi\neq 0$ and let 
		\[W(\xi)=\left\{\a=(a_1,\ldots,a_n)\in \A^*:\prod_{i=1}^{n}|r_{a_i}|\leq |\xi|^{-\delta}<\prod_{i=1}^{n-1}|r_{a_i}|\right\}.\]
		Recalling that $e(y)=e^{-2\pi i y}$ and using the self-similarity of $\mu$, we have the following:
		\begin{align}
			\label{e:splitFouriersum}
			\widehat{F\mu}(\xi)&=\int e(\xi F(x))\, d\mu(x)\nonumber\\
			&=\sum_{\a\in W(\xi)}p_{\a}\int e(\xi F(\psi_{\a}(x)))\, d\mu(x)\nonumber\\
			&=\sum_{\stackrel{\a\in W(\xi)}{\psi_{\a}([0,1])\cap \bigcup_{i=1}^{n}(x_i-C|\xi|^{-\delta},x_i+C|\xi|^{-\delta})= \varnothing}}p_{\a}\int e(\xi F(\psi_{\a}(x)))\, d\mu(x)\\
			&\phantom{=}+\sum_{\stackrel{\a\in W(\xi)}{\psi_{\a}([0,1])\cap \bigcup_{i=1}^{n}(x_i-C|\xi|^{-\delta},x_i+C|\xi|^{-\delta})\neq \varnothing}}p_{\a}\int e(\xi F(\psi_{\a}(x)))\, d\mu(x)\nonumber.
		\end{align}
		Observe that
		\[\widehat{(F\circ \psi_{\a})\mu}(\xi)=\int e(\xi F(\psi_{\a}(x)))\, d\mu(x).\]
		Applying the chain rule twice gives  
		\begin{equation}
			\label{e:chainrule}
		(F\circ \psi_{\a})'(x)=F'(\psi_{\a}(x))r_{\a}\quad \textrm{ and }\quad	(F\circ \psi_{\a})''(x)=F''(\psi_{\a}(x))r_{\a}^{2}.
		\end{equation}
		 It now follows from~\eqref{e:levelsetinclusion},~\eqref{e:chainrule} and the definition of $W(\xi)$ that whenever $|\xi|$ is sufficiently large, if $\a\in W(\xi)$ is such that $\psi_{\a}([0,1])\cap \bigcup_{i=1}^{n}(x_i-C|\xi|^{-\delta},x_i+C|\xi|^{-\delta})= \varnothing$, then \[|\xi|^{-\delta(k+2)}\preceq \min_{x\in [0,1]}|(F\circ \psi_{\a})''(x)|\leq \max_{x\in [0,1]}|(F\circ \psi_{\a})''(x)|\preceq 1, \] and 
		\[|\xi|^{-\delta(k+1)}\preceq \max_{x\in [0,1]}|(F\circ \psi_{\a})'(x)|\preceq 1.\] Using these bounds together with Corollary~\ref{t:self-similar}, we have that for $\a$ such that $\psi_{\a}([0,1])\cap \bigcup_{i=1}^{n}(x_i-C|\xi|^{-\delta},x_i+C|\xi|^{-\delta})= \varnothing$, 
		\begin{equation*}
			\left|\int e(\xi F(\psi_{\a}(x)))\, d\mu(x)\right|\preceq|\xi|^{\delta(2k+3)\kappa}|\xi|^{-\eta}\preceq |\xi|^{-\eta/2}
		\end{equation*} 
		(in the final inequality we used the definition of $\delta$). It follows that
		\begin{equation}
			\label{e:bound1}
			\left|\sum_{\stackrel{\a\in W(\xi)}{\psi_{\a}([0,1])\cap \bigcup_{i=1}^{n}(x_i-C|\xi|^{-\delta},x_i+C|\xi|^{-\delta})= \varnothing}}p_{\a}\int e(\xi F(\psi_{\a}(x)))\, d\mu(x)\right|\preceq |\xi|^{-\eta/2}.
			\end{equation}
			
		We bound the remaining term on the right hand side of~\eqref{e:splitFouriersum} as follows: 
		\begin{align}\label{e:bound2}
		&\left|	\sum_{\stackrel{\a\in W(\xi)}{\psi_{\a}([0,1])\cap \bigcup_{i=1}^{n}(x_i-C|\xi|^{-\delta},x_i+C|\xi|^{-\delta})\neq \varnothing}}p_{\a}\int e(\xi F(\psi_{\a}(x)))\, d\mu(x)\right|\nonumber\\
			&\leq \sum_{\stackrel{\a\in W(\xi)}{\psi_{\a}([0,1])\cap \bigcup_{i=1}^{n}(x_i-C|\xi|^{-\delta},x_i+C|\xi|^{-\delta})\neq \varnothing}}p_{\a}\nonumber\\
			&\leq \mu\left(\bigcup_{i=1}^{n}(x_i-(C+1)|\xi|^{-\delta},x_i+(C+1)|\xi|^{-\delta})\right)\nonumber\\
			&\preceq |\xi|^{-\delta s},
		\end{align} 
		where the $s>0$ appearing in the final line is the exponent whose existence is guaranteed by Lemma~\ref{l:FengFrostman}. Substituting~\eqref{e:bound1} and~\eqref{e:bound2} into the right hand side of~\eqref{e:splitFouriersum} implies our result. 
\end{proof}

\subsection{Proof of Theorem \ref{thm:analyticthm}}
    For an IFS of analytic maps $\Phi = \{\varphi_a \colon [0,1] \to [0,1]\}_{a \in \mathcal{A}}$, there is a dichotomy: either there exists an analytic diffeomorphism $F\colon [0,1]\to[0,1]$ and a self-similar IFS $\Psi\coloneqq \{\psi_{a}\}_{a\in \A}$ such that $\Phi=\{F\circ\psi_{a}\circ F^{-1}\}_{\a\in \A}$ (in this case we say that $\Phi$ is conjugate to self-similar), or $\Phi$ admits no such conjugacy. 
    The following lemma is basic and well known (for example it was clearly known to be true by the authors of \cite[Section~6]{AHW}), but we include the details for completeness. 
    \begin{lemma}\label{l:conjpush}
        Let $\Phi = \{\varphi_a \colon [0,1] \to [0,1]\}_{a \in \mathcal{A}}$ be an analytic IFS and assume there is an analytic diffeomorphism $F\colon [0,1]\to[0,1]$ and a self-similar IFS $\Psi\coloneqq \{\psi_{a}\}_{a\in \A}$ such that $\Phi=\{F\circ\psi_{a}\circ F^{-1}\}_{\a\in \A}$. 
        \begin{enumerate}
            \item\label{i:nonaffine} If there exists $a \in \mathcal{A}$ such that $\varphi_a$ is not affine then $F$ is not affine. 
            \item\label{i:pushbyconj} If $\mu$ is a self-conformal measure for $\Phi$ corresponding to the probability vector $\p=(p_{a})_{a\in \A}$, and $\nu$ is the self-similar measure for $\Psi$ also corresponding to $\p$, then $F\nu=\mu$. 
        \end{enumerate}
    \end{lemma}
    \begin{proof}
        \eqref{i:nonaffine}. If $F$ were affine then each $\varphi_a$ would be a composition of affine maps and would therefore be affine. 

        \eqref{i:pushbyconj}. Using the characterisation in~\eqref{e:definebernoullimeas}, for all Borel $A \subset [0,1]$ we have 
	    \begin{align*} 
		F\nu (A) = \nu(F^{-1} A) = \sum_{a \in \A} p_a \nu(\psi_{a}^{-1}\circ F^{-1} (A)) &= \sum_{a \in \A} p_a \nu (F^{-1}\circ\varphi_a^{-1} (A))\\*
		&=\sum_{a \in \A} p_a \varphi_{a}(F\nu) (A). 
		\end{align*} 
        In the third equality we have used that $\psi_{a}^{-1}\circ F^{-1}=F^{-1}\circ \varphi_{a}^{-1}$. This follows from the equation $F\circ \psi_{a}\circ F^{-1}=\varphi_{a}$. The equation above shows that $F\nu$ satisfies~\eqref{e:definebernoullimeas} for $\Phi$ and $\p$. However, $\mu$ is the unique probability measure satisfying~\eqref{e:definebernoullimeas}, so $F\nu=\mu$. 
    \end{proof}

    \begin{proof}[Proof of Theorem~\ref{thm:analyticthm}]
        If the IFS is not conjugate to self-similar, then the result follows from both \cite[Theorem~1.1]{AHW3} and \cite[Theorem~1.4]{BS}. We refer the reader to \cite[Section~6]{AHW3} and the discussion before Theorem~1.1 in \cite{BS} for a detailed explanation as to why the assumptions of these theorems are satisfied when there is no conjugacy. 

        If, on the other hand, the IFS is conjugate to self-similar, then by Lemma~\ref{l:conjpush} one can write $\mu = F\nu$ for some self-similar measure $\nu$ on $[0,1]$ and some non-affine analytic map $F \colon [0,1] \to [0,1]$. The result therefore follows from Theorem~\ref{thm:Analytic pushforward thm} in this case. 
    \end{proof}

	\section{Proof of Theorem \ref{t:main}}\label{s:mainthm}
	
		\subsection{Outline}\label{s:outline}
		We briefly outline the proof of Theorem~\ref{t:main}. We do so in the setting of Corollary~\ref{t:self-similar} (rather than Theorem~\ref{t:main}) for notational simplicity, and because most of the ideas in the proof of the general case are needed even in this simpler setting. 
		The first step in our proof will be to disintegrate our measure as 
	 \[\mu=\int_{\Omega} \mu_{\omega}\, dP(\omega).\] 
	Each $\mu_{\omega}$ will be a probability measure on $\mathbb{R}$. 
	The set $\Omega$ will be a suitable space of infinite sequences $([\a_i])_{i=1}^{\infty}$, where each $[\a_i]$ encodes a choice of homogeneous self-similar IFS. For notational simplicity we will denote a typical element in $\Omega$ by $\omega$, and $P$ will be a natural probability measure on $\Omega$. Crucially, a typical $\mu_{\omega}$ will resemble a self-similar measure for a well separated, homogeneous self-similar IFS. Moreover, each $\mu_{\omega}$ will have an infinite convolution structure. It is these properties which make $\mu_{\omega}$ easier to analyse than~$\mu$. 
	We expect that this method of disintegrating a stationary measure generated by an arbitrary countable IFS will be of independent interest. This idea of disintegrating a stationary measure in terms of a family of simpler measures was introduced by Galicer, Saglietti, Shmerkin and Yavicoli~\cite{GSSY}. In this paper they showed that an arbitrary self-similar measure can be expressed as an integral of random measures where each measure in the integral can be described as an infinite convolution of Dirac masses. This technique has also been applied in \cite{SSS,KaeOrp,SolSpi}. 
	Algom, the first named author, and Shmerkin~\cite{ABS} showed that one can express an arbitrary self-similar measure in terms of an integral of random measures where each measure resembles a self-similar measure for a well separated IFS. 

	Perhaps the most challenging step in our proof is establishing Proposition~\ref{p:decayoutsidesparse}, which asserts that for a set of $\omega$ with large $P$-measure, the Fourier transform of $\mu_{\omega}$ decays outside of a sparse set of frequencies. 
		We use some ideas from Kaufman~\cite{Kau}, and Mosquera and Shmerkin \cite{MS}, but the presence of infinitely many maps and maps with different contraction ratios causes substantial additional difficulties. 
		Indeed, we make extensive use of tools from large deviations theory, namely Hoeffding's inequality and Cram\'er's theorem, to prove that for a set of $([\a_i])_{i=1}^{\infty}$ with large $P$-measure, most elements in the sequence $([\a_i])_{i=1}^{\infty}$ will correspond to a self-similar IFS that is well-separated and contains many similarities. 
		Using the assumption~\eqref{e:largedevassump}, we can prove that with large probability the average of the contraction ratios corresponding to the entries in $([\a_i])_{i=1}^{\infty}$ will not be too small, and that most individual elements in $([\a_i])_{i=1}^{\infty}$ will not have a corresponding contraction ratio that is very small. 
		Since $\mu_{\omega}$ has an infinite convolution structure, $\widehat{\mu_\omega}$ has an infinite product structure. Under the above assumptions on $\omega$, for a large set of frequencies, a significant proportion of the important terms in this infinite product will have absolute value strictly less than $\upsilon$ for some uniform $\upsilon\in(0,1)$. 
		This is because each term in this product is the average of points on the unit circle in the complex plane, and a non-trivial combinatorial argument bounds the number of strings of indices of $([\a_i])$ such that for each index in such a string, all corresponding points on the circle are very close together (which is the only way that it is possible for there not to be decay). 
		This combinatorial argument is sometimes known as an \emph{Erd\H{o}s--Kahane argument}, after~\cite{Erdos2,Kah}. 
			
	The next step is to use Proposition~\ref{p:decayoutsidesparse} to prove Proposition~\ref{p:fmuomegadecay}, which asserts that the magnitude of the Fourier transform of $|\widehat{F \mu_{\omega}}|$ can be bounded above by some power of the frequency multiplied by some constant depending upon the first and second derivative of $F$. 
	To do so, after again using the convolution structure of $\mu_{\omega}$ and Taylor expanding $F$, we need to estimate the $\mu_{\omega}$-measure of a set of points whose image under $F'$, once rescaled, lies in the exceptional set of frequencies from Proposition~\ref{p:decayoutsidesparse}. 
	This is done by using large deviations theory to prove that with large $P$-measure, $([\a_i])_{i=1}^{\infty}$ will contain lots of indices corresponding to well-separated maps with large contraction ratios, so $\mu_{\omega}$ will have a uniform positive Frostman exponent. 
	
	We conclude the proof by observing that there exist $\eta,\delta>0$ such that for all $\xi\neq 0$ we can find a set $\Good\subset \Omega$ (given by Propositions~\ref{p:decayoutsidesparse} and~\ref{p:fmuomegadecay}) such that $P(\Good^{c}) \preceq \xi^{-\delta}$ and such that for all $\omega\in \Good$ we have $|\widehat{F\mu_{\omega}}(\xi)| \preceq |\xi|^{-\eta}$. 
	Then 
	\[|\widehat{F\mu}(\xi)|\leq \int_{\Good}|\widehat{F\mu_{\omega}}(\xi)|\, dP(\omega)+P(\Good^c)\preceq |\xi|^{-\delta}+|\xi|^{-\eta},\] which will complete our proof. 

	\subsection{Disintegrating the measure} 	
	For the rest of this section we fix a CIFS $\Phi=\{\varphi_{j,l}\}_{j\in J,l\in L_j}$ and a probability vector $\p=(p_{j,l})_{j\in J,l\in L_j}$ satisfying the hypotheses of Theorem~\ref{t:main}. To simplify our notation, we let $\A=\{(j,l)\}_{j\in J,l\in L_j}$. 
   Since the fibre IFS $\Psi_{j^*}$ is non-trivial, let the maps $\gamma_{l,j^*}$ and $\gamma_{l',j^*}$ have distinct fixed points, and consider the compositions 
     \begin{equation*}
        \overbrace{\gamma_{l,j^*}\circ \cdots \circ \gamma_{l,j^*}}^{\times n}\circ \overbrace{\gamma_{l',j^*}\circ \cdots \circ \gamma_{l',j^*}}^{\times n} \quad \mbox{ and } \quad \overbrace{\gamma_{l',j^*}\circ \cdots \circ \gamma_{l',j^*}}^{\times n}\circ \overbrace{\gamma_{l,j^*}\circ \cdots \circ \gamma_{l,j^*}}^{\times n}
    \end{equation*}
    for $n \in\mathbb{N}$. These two compositions have the same contraction ratio and associated probability vector, and when $n$ is sufficiently large, the set $[0,1]$ has disjoint images under these two compositions. 
    Note that for all $n' \in \N$, the set of $n'$-fold compositions of the maps in the CIFS gives another CIFS. Moreover, the stationary measure $\mu$ can be realised as a stationary measure for this new CIFS with respect to a different probability vector. 
    Therefore we henceforth assume without loss of generality that there exist two maps $\gamma_{l_1,j^*},\gamma_{l_2,j^*}$ in the original fibre IFS $\Psi_{j^*}$ such that 
	\begin{equation}\label{e:separated} 
		\gamma_{l_1,j^*}([0,1])\cap \gamma_{l_2,j^*}([0,1])=\varnothing.
	\end{equation}
	We may also assume that the contraction ratios of $\gamma_{l_1,j^*}$ and $\gamma_{l_2,j^*}$ are equal (denote the common value by $r_{*}$), and for the underlying probability vectors we have $p_{l_1,j^*}=p_{l_2,j^*}$ (denote the common value by $p_{*}$). 
    It follows trivially from~\eqref{e:separated} that there exists $c>0$ such if $\gamma_{l_1,j^*}(x)=r_{*} x + t_{l_1,j^{*}}$ and $\gamma_{l_2,j^*}(x)=r_{*} x + t_{l_2,j^{*}}$ then 
	\begin{equation}\label{e:csep}
	|t_{l_2,j^{*}}-t_{l_1,j^{*}}|\geq c.
	\end{equation}
	
	Our disintegration is defined using words of length $k$, i.e. $\A^{k}$. Here, $k\in \mathbb{N}$ is some fixed parameter that we will eventually take to be sufficiently large in our proof of Theorem~\ref{t:main}. Given $\a\in \A^{*}$, we let $\psi_{\a}\colon [0,1]^{d}\to[0,1]^{d}$ and $\gamma_{\a} \colon [0,1]\to [0,1]$ be the maps defined implicitly via the equation \[\varphi_{\a}(x_1,\ldots,x_{d+1})=(\psi_{\a}(x_1,\ldots,x_d),\gamma_{\a}(x_{d+1})).\] It follows from the fact that our CIFS is a fibre product IFS that for every $\a\in \A^*$, the map $\psi_{\a}$ only depends upon the $j$-component of each entry in $\a$. We emphasise that the map $\gamma_{\a}$ is always an affine contraction. 
	
	Given $\a=(a_1,\dotsc,a_k)$, $\b=(b_1,\dotsc,b_k)\in \A^k$, we write $\a\sim \b$ if for all $i$ such that $a_i\in \{(j_{*},l_1),(j_{*},l_2)\}$ we have $b_i\in \{(j_{*},l_1),(j_{*},l_2)\}$, and for all $i$ such that $a_i\notin \{(j_{*},l_1),(j_{*},l_2)\}$ we have $a_i=b_i$. 
	Clearly $\sim$ defines an equivalence relation on $\A^k$. For each $\a\in \A^k$ we let $[\a]$ denote the equivalence class of $\a$. The following lemma records some useful properties of this equivalence relation.
	\begin{lemma}
		\label{l:basic properties}
		The equivalence relation $\sim$ on $\A^{k}$ satisfies the following properties:
		\begin{enumerate}
			\item[1.] If $\a\sim \b$ and $\a\neq \b$, then $\gamma_{\a}([0,1])\cap \gamma_{\b}([0,1])=\varnothing$.%
			\item[2.] If $\a \sim \b$ then $\psi_{\a}=\psi_{\b}$.
			\item[3.] If $\a\sim \b$ then $\gamma_{\a}'=\gamma_{\b}'$ and $p_{\a} = p_{\b}$. 
			\item[4.] For all $\a\in \A^{k}$ we have $\# [\a]=2^{\#\{1\leq i\leq k:a_i\in \{(j_{*},l_1),(j_{*},l_2)\}\}}$.
		\end{enumerate}
	\end{lemma}
	\begin{proof}
		This lemma follows immediately from the definition of our equivalence relation and the properties of $\gamma_{l_1,j^*}$ and $\gamma_{l_2,j^*}$ stated above.
	\end{proof}
	It is useful to think of an equivalence class $[\a]$ as an IFS in its own right. Given an equivalence class $[\a]$, Lemma~\ref{l:basic properties}.3. states that each element in $[\a]$ has the same contraction ratio. We denote this common contraction ratio by $r_{[\a]}$. Given $\a$ we let $\{t_{\b,[\a]}\}_{\b\in [\a]}$ be the set of real numbers such that if we write $\gamma_{\b,[\a]}(x)=r_{[\a]}x+t_{\b,[\a]}$ for each $\b \in [\a]$, then 
	\[ \{ \varphi_{\b} \}_{\b \in [\a]}= \{(\psi_{\a},\gamma_{\b,[\a]})\}_{\b\in [\a]}.\]
	We let $I=\{[\a] : \a \in \A^k \}$ and $\Omega=I^{\mathbb{N}}$. 
	It will sometimes be convenient to denote elements of $\Omega$ by $\omega = ([\a_i])_{i=1}^{\infty}$ for some arbitrary choice of representatives $\a_i$. Given such an $\omega$, we let $(\j_{i,\omega})\in (J^{k})^{\N}$ be the unique sequence such that if $\b\in [\a_i] $ then $\psi_{\b}=\psi_{\j_{i,\omega}}$ for all $i$. The existence of $(\j_{i,\omega})$ follows from Lemma~\ref{l:basic properties}.2. Recall that $J$ is the index set for the base CIFS (see Definition \ref{Def:Fibre product}). Given $\omega\in \Omega$ we also let 
    \begin{equation}
        \label{e:x_omega def}
        \x_{\omega}\coloneqq \lim_{n\to\infty}(\psi_{\j_{1,\omega}}\circ \cdots \circ \psi_{\j_{n,\omega}})(0).
    \end{equation}
    
	We let $\sigma\colon \Omega\to \Omega$ be the usual left shift map, i.e. $\sigma (\omega ) = ([\a_2],[\a_3],\ldots)$. 
	We define the probability vector $\mathbf{q}$ on $I$ according to the rule 
	\[q_{[\a]}=\sum_{\b \in [\a]}p_{\b} = p_{\a} \cdot 2^{\#\{1\leq i\leq k:a_i\in \{(j_{*},l_1),(j_{*},l_2)\}\}}.\] 
	The second equality holds because of properties $3.$ and $4.$ stated in Lemma~\ref{l:basic properties}. We let $P$ denote the corresponding infinite product measure (Bernoulli measure) on $\Omega$, i.e. the unique measure which satisfies 
	\[ P(\{ ([\a_i])_{i=1}^{\infty} \in \Omega : [\a_1] = [\b_1], \dotsc, [\a_n] = [\b_n]  \}) = \prod_{i=1}^n {q_{[\b_i]}} \]
	for all $n \in \N$ and $[\b_1],\dotsc,[\b_n] \in I$. 
	Moreover, for each $[\a] \in I$ we define the uniform probability vector $\hat{\textbf{q}}^{[\a]}=(\hat{q}_{\b}^{[\a]})_{\b\in [\a]}$ on $[\a]$, where \[ \hat{q}_{\b}^{[\a]}=\frac{1}{\#[\a]}=\frac{1}{2^{\#\{1\leq i\leq k:a_i\in \{(j_{*},l_1),(j_{*},l_2)\}\}}} \]
	for all $\b \in [\a]$. 
	It follows from Lemma~\ref{l:basic properties}.4. that $\hat{\textbf{q}}^{[\a]}$ is a probability vector. 
	
	We now fix $\omega=([\a_i])_{i=1}^{\infty}\in \Omega$. 
	Let $\Sigma_{\omega}\coloneqq \prod_{i=1}^{\infty} [\a_i]$ and let $m_{\omega} \coloneqq \prod_{i=1}^{\infty}\hat{\textbf{q}}^{[\a_i]}$ be the infinite product measure supported on $\Sigma_{\omega}$. 
	We let $\pi_{\omega} \colon \Sigma_{\omega} \to\mathbb{R}^{d}$ be given by 
	\[ \pi_{\omega}((\b_1,\b_2,\ldots )) \coloneqq \lim_{n\to\infty}(\varphi_{\b_1}\circ \cdots\circ \varphi_{\b_n})(0)=\left(\x_{\omega},\sum_{m=1}^{\infty}t_{\b_m,[\a_m]}\prod_{i=1}^{m-1}r_{[\a_i]} \right)\]
	with the convention that the empty product is~1.  
	The measure which will appear in the disintegration is~$\tilde{\mu}_{\omega} \coloneqq \pi_{\omega} m_{\omega}$. 
	Note that $\tilde{\mu}_{\omega}$ is the law of the random variable $\pi_{\omega}((X_1,X_2,\ldots))$, where each $X_i$ is chosen uniformly at random from $[\a_i]$ according to $\hat{\textbf{q}}^{[\a_i]}$. 
	Note that \[\tilde{\mu}_{\omega}=\delta_{\x_{\omega}}\times 	\mu_{\omega}\] where $\delta_{\x_{\omega}}$ is the Dirac mass as $\x_{\omega}$ and $\mu_{\omega}$ is the probability measure supported in $[0,1]$ given by the infinite convolution: 
	\begin{equation}\label{e:convstructure} 
		\mu_{\omega} = \ast_{i=1}^{\infty}\frac{1}{\# [\a_i]}\sum_{\b\in [\a_i]}\delta_{t_{\b,[\a_i]}\cdot \prod_{j=1}^{i-1}r_{[\a_j]}}. 
	\end{equation}
	The following proposition shows that we can reinterpret our stationary measure $\mu$ in terms of the measures $\tilde{\mu}_{\omega}$. To streamline the proof of this proposition we let $S_{\lambda} \colon \R \to \R$ be given by $S_{\lambda}(x) = \lambda x$ for each $\lambda\in \R. $
	\begin{prop}
		\label{p:disintegration prop}
		The following disintegration holds: 
		\[\mu=\int_{\Omega} \tilde{\mu}_{\omega}\, dP(\omega).\]
	\end{prop}
	\begin{proof}	
		The proof is similar to that of \cite[Theorem~1.2]{ABS}, but we give the details for completeness. We will show that the probability measure $\nu \coloneqq \int_{\Omega} \tilde{\mu}_{\omega}\, dP(\omega)$ satisfies the equation 
		\begin{equation}\label{e:nuselfsim}
		\nu = \sum_{\a \in \A^k} p_{\a} \cdot \varphi_{\a} \nu.
		\end{equation}
		Crucially, the stationary measure $\mu$ is the unique probability measure satisfying this equation, so once we have established~\eqref{e:nuselfsim} we can deduce that $\mu=\nu$, completing the proof. 
		
		We observe
		\begin{align*}
			&\int_{\Omega} \tilde{\mu}_{\omega}\, dP(\omega)\\
			 &= \int_{\Omega}\delta_{\x_{\omega}}\times \mu_{\omega}\, dP(\omega)\\
			&=\int_{\Omega}\delta_{\x_{\omega}}\times\left(\ast_{m=1}^{\infty}\frac{1}{\# [\a_m]}\sum_{\b\in [\a_m]}\delta_{t_{\b,[\a_m]}\cdot \prod_{i=1}^{m-1}r_{[\a_i]}}\right)\, dP(\omega)\\
			&=\int_{\Omega}\delta_{\x_{\omega}}\times\left(\frac{1}{\#[\a_1]}\sum_{\b\in [\a_1]}\delta_{t_{\b,[\a_1]}}\ast \left(\ast_{m=2}^{\infty}\frac{1}{\# [\a_m]}\sum_{\b\in [\a_m]}\delta_{t_{\b_m,[\a_m]}\cdot \prod_{i=1}^{m-1}r_{[\a_i]}}\right)\right)\, dP(\omega)\\
			&=\int_{\Omega}\delta_{\x_{\omega}}\times\left(\frac{1}{\#[\a_1]}\sum_{\b\in [\a_1]}\delta_{t_{\b,[\a_1]}}\ast S_{r_{[\a_1]}}\mu_{\sigma(\omega)}\right)\, dP(\omega)\\
			&=\int_{\Omega}\delta_{\x_{\omega}}\times\left(\frac{1}{\#[\a_1]}\sum_{\b\in [\a_1]}\gamma_{\b}\mu_{\sigma(\omega)}\right)\, dP(\omega)\\
			&=\int_{\Omega}\frac{1}{\#[\a_1]}\sum_{\b\in [\a_1]}\delta_{\x_{\omega}}\times\gamma_{\b}\mu_{\sigma(\omega)}\, dP(\omega)\\
			&=\int_{\Omega}\frac{1}{\#[\a_1]}\sum_{\b\in [\a_1]}\varphi_{\b}(\delta_{\x_{\sigma(\omega)}}\times\mu_{\sigma(\omega)})\, dP(\omega)\\
			&=\sum_{[\a]\in I}\int_{[\a]\times \Omega}\frac{1}{\#[\a]}\sum_{\b\in [\a]}\varphi_{\b}(\delta_{\x_{\sigma(\omega)}}\times\mu_{\sigma(\omega)})\, dP(\omega)\\
			&=\sum_{[\a]\in I}\sum_{\b\in [\a]}\frac{1}{\#[\a]}\int_{[\a]\times \Omega}\varphi_{\b}(\delta_{\x_{\sigma(\omega)}}\times\mu_{\sigma(\omega)})\, dP(\omega)\\
			&=\sum_{[\a]\in I}\sum_{\b\in [\a]}\frac{q_{[\a]}}{\#[\a]}\int_{ \Omega}\varphi_{\b}(\delta_{\x_{\omega}}\times\mu_{\omega})\, dP(\omega)\\
			&=\sum_{[\a]\in I}\sum_{\b\in [\a]}p_{\b}\int_{ \Omega}\varphi_{\b}(\delta_{\x_{\omega}}\times\mu_{\omega})\, dP(\omega)\\
			&=\sum_{\b\in \A^k}p_{\b}\int_{ \Omega}\varphi_{\b}(\delta_{\x_{\omega}}\times\mu_{\omega})\, dP(\omega).
		\end{align*}
		In the penultimate line we used that $\frac{q_{[\a]}}{\#[\a]}=p_{\b}$ for all $\b\in [\a]$. Summarising the above, we have shown that the measure $\nu \coloneqq \int_{\Omega} \tilde{\mu}_{\omega}\, dP(\omega)$ satisfies~\eqref{e:nuselfsim}, which by our earlier remarks completes the proof. 
	\end{proof}
	
	Now suppose $\mu$ is the stationary measure for a CIFS on $\mathbb{R}$ with probability vector $\mathbf{p}$. By iterating, we may assume without loss of generality that our CIFS has two well-separated maps $\varphi_1$ and $\varphi_2$ with equal contraction ratios and probabilities. 
	As above, we can define an equivalence relation on words of length $k$ on the alphabet of our CIFS, and a probability measure $P$ on infinite sequences of equivalence classes. 
	The probability measures $\mu_{\omega}$ can then be defined as in~\eqref{e:convstructure}. 
	\begin{cor}\label{c:disintegrateinreals}
	If $\mu$ is the stationary measure for a CIFS $\Phi$ on $\R$ with probability vector $\p$, and $\mu$ is not supported on a singleton, then  
	\[ \mu = \int_{\Omega} \mu_{\omega} \, dP(\omega). \]
	\end{cor}
	\begin{proof}
	Consider the fibre product CIFS where $\Psi = \psi_0$ is the single-element self-similar IFS on $\R$ consisting of the map $\psi_0(x) = x/2$ and whose corresponding fibre is $\Phi$. 
	Let $\mu'$ be the stationary measure for the fibre product CIFS. Then by Proposition~\ref{p:disintegration prop}, 
	\[ \delta_0 \times \mu = \mu' = \int_{\Omega} (\delta_0 \times \mu_{\omega}) \, dP(\omega) = \delta_0 \times \int_{\Omega} \mu_{\omega} \, dP(\omega), \]
	which completes the proof. 
	\end{proof}

	\subsection{Large deviations and Frostman exponent}\label{s:largedev}

	We begin this section by proving a large deviation bound. This bound provides two useful properties. Loosely speaking, the first property ensures that with high probability, a sequence $([\a_i])_{i=1}^{\infty}\in \Omega$ will be such that for all $N\geq N'$, $[\a_i]$ contains many elements for most values $1\leq i\leq N$ (see the definition of $\tilde{\Omega}_{1}$). 
	The second property will ensure that with high probability, after a time $N'$, for a sequence $([\a_i])_{i=1}^{\infty}\in \Omega$ we will be able to control the average behaviour of the contraction ratios $r_{[\a_i]}$ when viewed at a scale $N\geq N'$ (see the definition of $\tilde{\Omega}_{2}$, $\tilde{\Omega}_{3}$, and $\tilde{\Omega}_{4}$). This large deviation bound will be used in the proof of Proposition~\ref{p:frostman} regarding the Frostman exponent of a typical $\mu_{\omega}$, and Proposition~\ref{p:decayoutsidesparse} which guarantees polynomial Fourier decay for $\mu_{\omega}$ outside of a sparse set of frequencies with high probability. 
	
	Recall that $p_{*}$ is the probability weight associated to the maps $\varphi_{j^*,l_1}$ and $\varphi_{j^*,l_2}$, and that $\Lambda$ is the Lyapunov exponent defined in~\eqref{e:lyapunov}. 
	Given $k,N' \in \N$, and $\alpha>0$, and using the convention that the empty product is~$1$, let 
	\begin{align*} 
		\tilde{\Omega}_1 &\coloneqq\bigcap_{N=N'}^{\infty}\left\{([\a_i])_{i=1}^{\infty}\in \Omega:\#\{1\leq i \leq N: \# [\a_i] > 2^{p_{*}k}\}\geq N (1-e^{-\alpha k}) \right\}, \\
		\tilde{\Omega}_2 &\coloneqq  \bigcap_{N=N'}^{\infty} \left\{ ([\a_i])_{i=1}^{\infty} \in \Omega : \prod_{i=1}^N |r_{[\a_i]}| > e^{-2\Lambda kN}  \right\} ,\\
		\tilde{\Omega}_3 &\coloneqq  \bigcap_{N=N'}^{\infty} \left\{([\a_i])_{i=1}^{\infty}\in \Omega:\#\{1\leq i \leq N: |r_{[\a_i]}| \geq \exp(-e^{3\alpha k/4}) \}\geq N (1-e^{-\alpha k}) \right\} , \\
		\tilde{\Omega}_4 &	\coloneqq \bigcap_{N=N'}^{\infty} \Bigg\{([\a_i])_{i=1}^{\infty}\in \Omega:\prod_{\substack{1 \leq i \leq N \\ |r_{[\a_i]}| < \exp(-e^{3\alpha k/4})}} |r_{[\a_i]}| \geq \exp\Big(2N\sum_{\stackrel{[\a]\in I}{|r_{[\a]}|<\exp(- e^{3\alpha k/4})}}q_{[\a]}\log|r_{[\a]}|\Big) \Bigg\}.
	\end{align*}
	The expression $\exp(-e^{3\alpha k /4})$ which appears in the definition of $\tilde{\Omega}_3$ may seem odd; it is a much smaller term than one would expect to see in a large deviation argument. However, taking this term instead of a more traditional large deviation term significantly simplifies part of our later arguments, and in particular simplifies the proof of Proposition~\ref{p:decayoutsidesparse}. 

	Define 
	\[ \Omega^* = \Omega_{k,\alpha,N'} \coloneqq \tilde{\Omega}_1 \cap \tilde{\Omega}_2 \cap \tilde{\Omega}_3 \cap \tilde{\Omega}_4. \] 
    In order to bound the measure of $\Omega^*$, we will make repeated use of the following large deviations inequality, which is a straightforward application of Hoeffding's inequality~\cite{Hoe}. 
    \begin{lemma}\label{l:hoeffding}
        Let $X_1,X_2,\dotsc$ be i.i.d. real-valued random variables, let $S \subset \R$, let $p \coloneqq \mathbb{P}(X_1 \in S)$, and let $q < p$. 
        Then for all $N \in \N$, 
        \[
        \mathbb{P}(\#\{n \in \{1,\dotsc,N\} : X_n \in S \} \leq N q ) \leq e^{-2(p-q)^2 N}.
        \]
    \end{lemma}
    \begin{proof}
        For all $n \in \N$ let $Y_n$ be the random variable taking value $0$ if $X_n \notin S$ and $1$ if $X_n \in S$. Then the $Y_n$ are i.i.d. with $\mathbb{E}(Y_1) = p$. The probability we are trying to calculate is precisely 
        \[
        \mathbb{P}\left( \Big(\sum_{n=1}^N Y_n\Big) - N \mathbb{E}(Y_1) \leq -(p-q)N \right). 
        \]
        By Hoeffding's inequality we can bound this above by $e^{-2(p-q)^2 N}$, as required. 
    \end{proof}

    We will also need the following version of Cram\'er's theorem. 
    \begin{lemma}[Section~2.7 in \cite{Dur}, page~508 in \cite{Klenke}]\label{l:cramer}
        Let $X_1,X_2,\dotsc$ be i.i.d. discrete real-valued random variables. Assume that $\mathbb{E}(e^{\tau X_1}) < \infty$ for some $\tau > 0$. Then for all $\delta > 0$ there exist $\gamma > 0$, $N_0 \in \N$ such that for all $N \geq N_0$, 
        \[ \mathbb{P}\left(|X_1 + \dotsb + X_N - N\mathbb{E}(X_1)|> N \delta\right) \leq e^{-\gamma N}. \] 
    \end{lemma}
We are now ready to bound the measure of $\Omega^{*}$.
	\begin{prop}\label{p:largedev}
		Let $\Phi$ and $\p$ be a CIFS and a probability vector satisfying the hypotheses of Theorem~\ref{t:main}. There exists $\alpha > 0$ such that for all $k\in \N$ sufficiently large, there exists $\beta_k > 0$ such that for all $N'\in \mathbb{N}$,  
		\[1- \mathbb{P}( \Omega^* ) \preceq_{k}  e^{-\beta_k N'}. \]
	\end{prop}%
	\begin{proof}
		Throughout the proof we let $N' \in \N$ be arbitrary and fixed. The probability of the event $a_i\in \{(j_{*},l_1),(j_{*},l_2)\}$ is $2p_{*}$ for all $1\leq i\leq k$. Therefore applying Lemma~\ref{l:hoeffding} there exists some small $\alpha>0$ such that for all integers $k$ we have
        \begin{align}
		\label{e:klargedev}
	 & \sum_{[\a] \in I:\, \# [\a] \leq 2^{p_{*} k} } q_{[\a]}\nonumber\\
            &=\sum_{\stackrel{\a\in \A^k}{\#\{i:a_i\in \{(j_{*},l_1),(j_{*},l_2)\}\} \leq p_{*} k}}p_{\a}\nonumber \\
            &=m\left(\left\{\a\in \A^{\N}:\#\{1\leq i\leq k:a_i\in \{(j_{*},l_1),(j_{*},l_2)\}\} \leq p_{*} k\right\}\right)\nonumber\\
            &\leq e^{-2\alpha k}.
		\end{align}
	Here $m$ denotes the infinite product measure on $\A^{\N}$ corresponding to the probability vector $\p$.	In the first equality we used Lemma~\ref{l:basic properties}.4. 
    
    We now fix $\alpha$ such that~\eqref{e:klargedev} holds. To complete our proof of this proposition, it suffices to show that for this value of $\alpha$, for each $\tilde{\Omega}_i$, for all $k$ sufficiently large there exists $\beta_{i,k}>0$ such that for all $N' \in \N$, 
		\begin{equation}
			\label{e:Omega_i STS}
			1-P(\tilde{\Omega}_i)\preceq_{k}e^{-\beta_{i,k}N'}.
		\end{equation}
		Taking $\beta_{k} \coloneqq \min\{\beta_{1,k},\beta_{2,k},\beta_{3,k},\beta_{4,k}\}$, our result will then follow. \bigskip
		
		\noindent \textbf{Verifying~\eqref{e:Omega_i STS} for $\tilde{\Omega}_1$.} 
	
			We begin by emphasising that~\eqref{e:klargedev} gives an upper bound for the probability of the event $\#[\a]\leq 2^{p_{*}k}$. Therefore if we let 
            \[\tilde{\Omega}_{k,\alpha}^{(N)}\coloneqq\left\{([\a_i])\in \Omega:\#\{1\leq i\leq N:\#[\a_i] > 2^{p_{*}k}\}\geq N (1-e^{-\alpha k})\right\},\] then we can apply Lemma~\ref{l:hoeffding} again to show that for all $k\in \N$ there exists $\beta_{1,k} > 0$ such that for all $N\in \mathbb{N}$ we have
			\[ P( \tilde{\Omega}_{k,\alpha}^{(N)} ) \geq 1 - e^{- \beta_{1,k} N}. \] 
			Using this inequality, we observe that 
			\[ 1-P(  \tilde{\Omega}_1 ) = P\left( \bigcup_{N=N'}^{\infty} (\Omega \setminus \tilde{\Omega}_{k,\alpha}^{(N)} ) \right) \leq \sum_{N=N'}^{\infty} e^{- \beta_{1,k} N} = \frac{e^{-\beta_{1,k} N'}}{1-e^{-\beta_{1,k}}} \preceq_{k} e^{-\beta_{1,k} N'}. \] 
			Therefore~\eqref{e:Omega_i STS} holds for $\tilde{\Omega}_1$.\bigskip
			
			\noindent 	\textbf{Verifying~\eqref{e:Omega_i STS} for $\tilde{\Omega}_2$.}
			
	First observe the following: 
		\begin{align}	
			\label{eq:Cramer setup}
			&P \left( \left\{ ([\a_i]) \in \Omega : \prod_{i=1}^N |r_{[\a_i]}| \leq e^{-2\Lambda kN}  \right\} \right) \nonumber \\&=\sum_{\stackrel{\a\in \A^{kN}}{\prod_{i=1}^{kN}|r_{a_i}|\leq e^{-2\Lambda kN}}}p_{\a}\nonumber\\
            &=m\left( \left\{ \a\in \A^{\N}:-\sum_{i=1}^{kN}\log|r_{a_i}|-\Lambda kN\geq \Lambda kN \right\} \right).            
		\end{align} 
		Since $\sum_{j\in J}\sum_{l\in L_{j}}p_{l,j}|r_{l,j}|^{-\tau}<\infty$ for some $\tau>0$, we can apply Lemma~\ref{l:cramer}, considering the $-\log |r_{a_i}|$ as i.i.d. random variables with expectation $\Lambda$.  Applying this lemma to the expression in~\eqref{eq:Cramer setup}, we see that there exists $\delta>0$ and $N''\in\mathbb{N}$ such that for all $N \geq N''$ we have
		\begin{equation*}\label{e:jordansahlsten} 
			P \left( \left\{ ([\a_i]) \in \Omega : \prod_{i=1}^N |r_{[\a_i]}| \leq e^{-2\Lambda kN}  \right\} \right) \leq e^{-\delta kN} . 
		\end{equation*} 
		For such a value of $\delta$ we let $\beta_{2,k} \coloneqq \delta/2$. If $N' \geq N''$ is sufficiently large, the following string of inequalities holds: 
		\begin{align*} P\left( \bigcup_{N=N'}^{\infty} \left\{ ([\a_i]) \in \Omega : \prod_{i=1}^N |r_{[\a_i]}| \leq e^{-2\Lambda k N}  \right\} \right)  &\leq e^{-\delta kN'}\sum_{i=0}^{\infty}e^{-\delta ki}\\
			& \leq e^{-\beta_{2,k} kN'}\\
			& \leq e^{-\beta_{2,k} N'}.  
			\end{align*}%
		Thus the desired bounds holds for all $N'$ sufficiently large. 
        Therefore we can ensure that~\eqref{e:Omega_i STS} holds for $\tilde{\Omega}_{2}$ for all $N' \in \N$. \bigskip
			
			\noindent \textbf{Verifying~\eqref{e:Omega_i STS} for $\tilde{\Omega}_3$.}
			
			We begin by remarking that for $\a\in \A^{k}$ the condition  $|r_{\a}| \geq \exp(- e^{3\alpha k/4})$ is equivalent to $\sum_{i=1}^{k}\log |r_{a_i}|\geq -e^{3\alpha k/4}$. For all $k$ sufficiently large, the following holds: 
			\begin{align*}
				\sum_{\stackrel{[\a]\in I}{|r_{[\a]}| < \exp(- e^{3\alpha k/4})}}q_{[\a]}&=\sum_{\stackrel{\a\in \A^{k}}{|r_{\a}| < \exp(- e^{3\alpha k/4})}}p_{\a} \\
				&=\sum_{\stackrel{\a\in \A^{k}}{-\sum_{i=1}^{k}\log |r_{a_i}|< e^{3\alpha k/4}}}p_{\a}\\
				&\leq \sum_{\stackrel{\a\in \A^{k}}{(\sum_{i=1}^{k}\log |r_{a_i}|+k\Lambda)^2> (-e^{3\alpha k/4}+k\Lambda)^2}}p_{\a}\\
				&\leq \frac{ \Var \left[ \sum_{i=1}^{k}\log|r_{a_i}|\right]}{(-e^{3\alpha k/4}+k\Lambda)^2}\\
				&\leq \frac{ 2\Var \left[ \sum_{i=1}^{k}\log|r_{a_i}|\right]}{e^{3\alpha k/2}}\\
				&=\frac{2k\Var [\log |r_{a}|]}{e^{3\alpha k/2}}\\
				&\leq e^{-5\alpha k/4}.
			\end{align*} 
			In the second inequality we used Markov's inequality. In the penultimate line we used that the random variable $\a\mapsto \sum_{i=1}^{k}\log |r_{a_{i}}|$ is the sum of independent random variables and  therefore $\Var\left[\sum_{i=1}^{k}\log|r_{a_i}|\right]=k\Var[\log |r_{a}|]$. In the final line we used that $\Var[\log |r_{a}|]<\infty$, which is a consequence of~\eqref{e:largedevassump}. 
            
            Summarising, we have shown that 
			\begin{equation}
				\label{e:omega3 bound}
				\sum_{\stackrel{[\a]\in I}{|r_{[\a]}| < \exp(- e^{3\alpha k/4})}}q_{[\a]}\leq e^{-5\alpha k/4}
			\end{equation} 
			for all $k$ sufficiently large.
		 Equation~\eqref{e:omega3 bound} gives an upper bound for the probability of the event $|r_{[\a_i]}| < \exp(-e^{3\alpha k/4})$. As such we can apply Lemma~\ref{l:hoeffding} in an analogous way to our argument for $\tilde{\Omega}_{1}$ to show that for all $k$ sufficiently large there exists $\beta_{3,k}>0$ such that 
			\[ 1-P(\tilde{\Omega}_3)\preceq_{k}e^{-\beta_{3,k}N'}\] 
			holds for all $N' \in \N$. Therefore~\eqref{e:Omega_i STS} holds for $\Omega_{3}$. \bigskip
			
			\noindent \textbf{Verifying~\eqref{e:Omega_i STS} for $\tilde{\Omega}_{4}$.}	
			
			Consider the random variable $f\colon I\to \R$ given by 
			\begin{align*}
				f([\a]) &=
				\begin{cases}
					0        & \text{if } |r_{[\a]}| \geq  \exp(- e^{3\alpha k/4}), \\
					\log |r_{[\a]}|        & \text{if } |r_{[\a]}|< \exp(- e^{3\alpha k/4}).
				\end{cases}
			\end{align*}
			The significance of the random variable $f$ is that for all $([\a_i])\in \Omega$ we have  
			\[ \prod_{\substack{1 \leq i \leq N \\ |r_{[\a_i]}| < \exp(- e^{3\alpha k/4})}} |r_{[\a_i]}|=\exp\left({\sum_{i=1}^{N}f([\a_i])}\right).\] 
			It follows from our underlying assumptions 
			\[ -\sum p_{l,j}\log |r_{l,j}|<\infty\quad \text{ and }\quad\sum p_{l,j} |r_{l,j}|^{-\tau}<\infty \] 
			for some $\tau>0$, that 
			\[ \sum_{[\a]\in I}q_{[\a]}f([\a])<\infty\quad \text{ and }\quad \sum_{[\a]\in I}q_{[\a]}e^{-\tau f([\a])}<\infty \] for the same value of $\tau$. Therefore the random variable $f$ satisfies the assumptions of Lemma~\ref{l:cramer}. We also remark that 
			\[
			\mathbb{E}[f]=\sum_{[\a]\in I}q_{[\a]}f([\a])=\sum_{\stackrel{[\a]\in I}{|r_{[\a]}|<\exp(- e^{3\alpha k/4})}}q_{[\a]}\log|r_{[\a]}|.
            \]
			Now applying Lemma~\ref{l:cramer} and replicating the argument given in the proof of~\eqref{e:Omega_i STS} for $\tilde{\Omega}_2$ with the random variable $a\mapsto \log |r_{a}|$ replaced with the random variable $[\a]\mapsto f([\a])$, we see that there exists $\beta_{4,k}>0$ such that for all $N' \in \N$, the following probability can be bounded above by $e^{-\beta_{4,k}N'}$: 
			\begin{equation*} 
			P\Bigg(\bigcup_{N=N'}^{\infty}\Big\{([\a_i]):\prod_{\substack{1 \leq i \leq N \\ |r_{[\a_i]}| < \exp(- e^{3\alpha k/4})}} |r_{[\a_i]}|<\exp\Big(2N\sum_{\stackrel{[\a]\in I}{|r_{[\a]}|<\exp(- e^{3\alpha k/4})}}q_{[\a]}\log|r_{[\a]}| \Big)\Big\}\Bigg). 
			\end{equation*}
			This final statement is equivalent to~\eqref{e:Omega_i STS} for $\tilde{\Omega}_{4}$. 
			Our proof of Proposition~\ref{p:largedev} is therefore complete. 
		\end{proof}
		The expression appearing within the $\exp$ term in the definition of $\tilde{\Omega}_{4}$ might at first appear not to be particularly meaningful. The following lemma shows that in fact it can be controlled by our parameter $k$.
		
		\begin{lemma}
			\label{l:little o lemma}
			Let $\Phi$ and $\p$ be a CIFS and a probability vector satisfying the hypotheses of Theorem~\ref{t:main}. Then for $\alpha$ as in the statement of Proposition~\ref{p:largedev}, we have that
			\[	-\sum_{\stackrel{[\a]\in I}{|r_{[\a]}|<\exp(- e^{3\alpha k/4})}}q_{[\a]}\log|r_{[\a]}|=o_{k}(1).\]	
		\end{lemma} 
		\begin{proof}
			Let $\alpha$ be as in the statement of Proposition~\ref{p:largedev}. Replicating the argument used to prove~\eqref{e:omega3 bound} we can show that for all $k$ sufficiently large and any $j \geq k$ we have 
			\begin{equation}
				\label{e:general measure bound}
				\sum_{\stackrel{[\a]\in I}{|r_{[\a]}| < \exp(- e^{3\alpha j/4})}}q_{[\a]}\leq e^{-5\alpha j/4}.
			\end{equation}
			Therefore the following holds for all $k$ sufficiently large: 
			\begin{align*}
				&-\sum_{\stackrel{[\a]\in I}{|r_{[\a]}|<\exp(- e^{3\alpha k/4})}}q_{[\a]}\log|r_{[\a]}|\\
				&= -\sum_{j=k}^{\infty}\sum_{\stackrel{[\a]\in I}{\exp(- e^{3\alpha (j+1)/4})\leq |r_{[\a]}|<\exp(-  e^{3\alpha j/4})}}q_{[\a]}\log|r_{[\a]}| \\
			&	\leq e^{3\alpha/4} \sum_{j=k}^{\infty}\sum_{\stackrel{[\a]\in I}{\exp(-  e^{3\alpha (j+1)/4})\leq |r_{[\a]}|<\exp(-  e^{3\alpha j/4})}}q_{[\a]}e^{3\alpha j/4}\\
				&\leq e^{3\alpha/4} \sum_{j=k}^{\infty}e^{3\alpha j/4}\sum_{\stackrel{[\a]\in I}{ |r_{[\a]}|<\exp(-  e^{3\alpha j/4})}}q_{[\a]}\\
				&\leq e^{3\alpha/4} \sum_{j=k}^{\infty}e^{-\alpha j/2}=o_{k}(1),
			\end{align*}
			where in the final line we applied~\eqref{e:general measure bound}.
		\end{proof}

		The following proposition is an application of Proposition~\ref{p:largedev}. 
		It says that outside of a set of $\omega$ whose $P$-measure decays in a way that depends upon $k$, the Frostman exponent of $\mu_{\omega}$ can be bounded from below by a quantity that only depends upon our initial IFS and the underlying probability vector. 
        Recall the Lyapunov exponent $\Lambda$ from~\eqref{e:lyapunov}. 
		Here and elsewhere, `sufficiently large' allows dependence on the IFS and the measure $\mu$ only. 
		\begin{prop}\label{p:frostman}
			For all $k\in\mathbb{N}$ sufficiently large, there exists $\beta \in (0,1]$ such that for all $r'\in(0,\infty)$, there exists $\Omega_{1}\subset \Omega$ such that $P(\Omega \setminus \Omega_1) \preceq_{k} (r')^{\beta}$ and such that for all $\omega \in \Omega_1$, $r \in (0,r')$ and $x \in \R$,
			\[ \mu_{\omega}((x,x+r)) \leq 3r^{s_{\Phi}}, \]
			where 
			\[ s_{\Phi} \coloneqq \frac{p_{*}\log 2}{5\Lambda } > 0. \] 
		\end{prop}
		\begin{proof}
			Let $\alpha>0$ be as in Proposition~\ref{p:largedev}. Let $k\in\mathbb{N}$ be sufficiently large that $e^{-\alpha k}<1/2$ and let $\beta_{k}$ be the associated parameter coming from Proposition~\ref{p:largedev}. We let $\beta = \min\{ \beta_k/(3k \Lambda), 1\}$. 
			If $r'\geq 1$ then the result is true simply by letting $\Omega_1 = \varnothing$, so we let $r'\in(0,1)$ be arbitrary. 
			
			For any $r\in(0,1)$ we let $N^{(r)} \coloneqq \lfloor \frac{-\log r}{2k\Lambda} \rfloor$. 
			Letting $N' = N^{(r')}$, we get a set $\Omega^*$ from Proposition~\ref{p:largedev}, which we define to be $\Omega_1$. 
			Observe that
			\[ P(\Omega \setminus \Omega_1) \preceq_{k} e^{-\beta_k N^{(r')}} \preceq_{k}  (r')^{\beta}. \]  
			Now fix any $\omega \in \Omega_1$ and $r \in (0,r')$.
			Since $\omega\in \Omega_{1}$ we also have $\omega\in \tilde{\Omega}_{1}$. 
			Appealing now to the definition of $\tilde{\Omega}_{1}$, and recalling parts $1.$, $3.$ and $4.$ of Lemma~\ref{l:basic properties}, we see that $\mu_{\omega}$ is supported inside at least $2^{p_{*} k N_r/2}$ disjoint intervals which each have length $\prod_{n=1}^{N^{(r)}}| r_{[\a_n]}|$ and mass at most $2^{- p_{*} k N_r/2}$. 
			Here we have used our assumption that $k$ is sufficiently large so that $e^{-\alpha k}<1/2$. 
			Since $\omega \in \Omega_{1}$ and therefore $\omega\in \tilde{\Omega}_{2}$, we have $\prod_{n=1}^{N^{(r)}}| r_{[\a_n]}| \geq \exp(-2\Lambda k N^{(r)}) \geq r$. 
			Using this inequality and the above, it follows that for all $x \in \R$, 
			\[ \mu_{\omega}((x,x+r)) \leq 3 \cdot  2^{- p_{*} k N_r/2} \leq 3r^{s_{\Phi}}. \qedhere \]
		\end{proof}
		
		\subsection{Decay outside sparse frequencies} 
		We will use the following basic fact about weighted sums of points on the unit circle. 
		\begin{lemma}
			\label{l:fourierdecay}
			Let $\p=(p_1,\ldots,p_n)$ satisfy $p_i>0$ for all $i$ and $\sum_i p_i=1$. Then for all $\delta \in (0,\pi]$ there exists $\upsilon\in(0,1)$ (depending upon $\p$ and $\delta$) such that if the points $z_1,\ldots,z_n\in \mathbb{C}$ each satisfy $|z_i|=1$, and if there exist $j,k$ such that $\mbox{dist}(arg(z_j)-arg(z_k),2\pi\mathbb{Z}) \geq \delta$, then $|\sum_i p_i z_i|\leq \upsilon$. 
		\end{lemma}
		\begin{proof}
		This lemma is straightforward to verify, and is left to the reader. 
		\end{proof}
		
		It follows from~\eqref{e:convstructure} that for all $\omega\in \Omega$, the Fourier transform of $\mu_{\omega}$ is the infinite product 
		\begin{equation}\label{e:ft}
			\widehat{\mu_{\omega}}(\xi)=\prod_{m=1}^{\infty}\frac{1}{\# [\a_m]}\sum_{\b\in [\a_m]}e\Big(\xi \cdot t_{\b,[\a_m]}\cdot \prod_{i=1}^{m-1}r_{[\a_i]} \Big).
		\end{equation} 
		
		We are now ready to prove the key technical result that if we fix $k$ large enough, for a set of $\omega$ with large $P$-measure, the Fourier transform of $\mu_{\omega}$ decays at a polynomial rate outside of a sparse set of frequencies. 	
		To prove this result we will use~\eqref{e:ft} which connects the behaviour of $\widehat{\mu_{\omega}}(\xi)$ to the distribution of a sequence modulo one depending on $\xi$, namely the sequence given by the terms on the left-hand side of~\eqref{e:eksequence} below. In particular, the failure of Fourier decay at a certain frequency means that this sequence spends a disproportionate amount of time being close to $0$ or $1$. This observation is what is exploited in the classical Erd\H{o}s--Kahane argument to prove Fourier decay outside a sparse set of frequencies. 
        However, this argument will not work for all $\omega = ([\a_i])_{i=1}^{\infty} \in \Omega$, because it is possible that a disproportionate number of the $[\a_i]$ will be a single element set. To overcome this issue we use the large deviation results from Section~\ref{s:largedev} to show that this problem can only occur for a small set of $P$-measure. Therefore we can successfully apply the Erd\H{o}s--Kahane argument for a set of large $P$-measure, giving the desired result. 
		\begin{prop}\label{p:decayoutsidesparse}
			For all $k\in\mathbb{N}$ sufficiently large there exist $\epsilon, C_k> 0$ such that for all $T'>0$ there exists $\Omega_2 \subset \Omega$ such that $P(\Omega \setminus \Omega_2) \preceq_{k} (T')^{-\epsilon}$, and such that for all $T \geq T'$ and $\omega \in \Omega_2$, the set $\{\xi\in [-T,T] : |\widehat{\mu_{\omega}}(\xi)|\geq T^{-\epsilon} \}$ can be covered by at most $C_{k} T^{o_{k}(1)}$ intervals of length~$1$.
		\end{prop}	
		\begin{proof}
			The proof of this proposition is long so we split it into more manageable parts. \bigskip
			
			\noindent \textbf{Part 1. Defining $\Omega_{2}$ and introducing our strategy of proof.}
			
			Let $\alpha$ be as in Proposition~\ref{p:largedev}. Let $k\in \N$ be sufficiently large so that Proposition~\ref{p:largedev} applies and let $\beta_{k}>0$ be the corresponding parameter. 
			It clearly suffices to prove the proposition for $T' > e^{2\Lambda k}$, so fix an arbitrary such $T'$, and let $N' \in \N$ be such that $e^{2\Lambda kN'} < T' \leq e^{2\Lambda k(N'+1)}$. 
			Let $\Omega_{2}\coloneqq \Omega^{*}$ for these values of $k$, $\alpha$ and $N'$. 
			If we let $\epsilon' \coloneqq \beta_k/(2k\Lambda)$, then it follows from Proposition~\ref{p:largedev} and the definition of $N'$ that 
			\[ P(\Omega \setminus \Omega_2)  \preceq_{k} e^{-\beta_k N'} \preceq_{k} (T')^{-\epsilon'}.\]  
			
			We fix $\omega \in \Omega_2$ and $T \geq T'$. We now set out to show that for some $\epsilon>0$ depending only upon $k$, the set of frequencies $\xi\in [-T,T]$ for which $|\widehat{\mu_{\omega}}(\xi)|\geq T^{-\epsilon}$ can be covered by $C_{k} T^{o_{k}(1)}$ intervals of length~$1$. If we can do this, then without loss of generality we can assume $\epsilon \leq \epsilon'$, and the proof will be complete. 
			
			Let $N \in \N$ be such that $e^{2\Lambda k(N-1)} < T \leq e^{2\Lambda kN}$, noting that $N \geq N'$. Without loss of generality we may assume $T = e^{2\Lambda kN}$. 
			Let $N_{\omega}\in\mathbb{N}$ be the minimal positive integer satisfying 
			\begin{equation}%
				\left|T \cdot \prod_{i=1}^{N_{\omega}+1}r_{[\a_i]}\right|<1.
			\end{equation}%
			It follows from the fact that our CIFS in uniformly contracting that there exists $c_{1}>1$ depending only on the underlying CIFS such that $N_{\omega}\leq c_{1}N$. Since $\omega\in \Omega_{2}$ and therefore $\omega\in \tilde{\Omega}_{2}$, it also follows that $N\leq N_{\omega}$. Combining these statements gives
			\begin{equation*}
				\label{e:N comparisons}
				N\approx N_{\omega}.
			\end{equation*} Using this expression together with the inequalities $e^{2\Lambda k(N-1)} < T \leq e^{2\Lambda kN}$ yields
			\begin{equation}
				\label{e:T and N comparison}
				\frac{\log T}{k}\approx N_{\omega}.
			\end{equation} 
			It follows from the definition of $N_{\omega}$ that 
			\[ \left|T \cdot \prod_{i=1}^{N_{\omega}}r_{[\a_i]}\right|\in [1,|r_{[\a_{N_{\omega}+1}]}|^{-1}].\] 
			Now using the fact that $\omega\in \Omega_{2}$, and therefore $\omega\in \tilde{\Omega}_{4}$, the above inclusion implies that 
			\begin{equation}
				\label{e:asymptotic inclusion}
				\left|T \cdot \prod_{i=1}^{N_{\omega}}r_{[\a_i]}\right|\in \Bigg[1,\max\Big\{\exp( e^{3\alpha k/4}),\exp\Big(-2N_{\omega}\sum_{\stackrel{[\a]\in I}{|r_{[\a]}|<\exp( -e^{3\alpha k/4})}}q_{\a}\log|r_{[\a]}|\Big)\Big\}\Bigg].
			\end{equation}
			
			Let $\xi \in [-T,T]\setminus \{0\}$. By~\eqref{e:ft}, 
			\begin{equation}\label{e:decaybound} 
				|\widehat{\mu_{\omega}}(\xi)|\leq \prod_{i=1}^{N_{\omega}} \frac{1}{\# [\a_i]} \left| \sum_{\b\in [\a_i]}e\Big(\xi \cdot t_{\b,[\a_i]}\cdot \prod_{j=1}^{i-1}r_{[\a_j]}\Big)\right|.
			\end{equation}
			Let 
			\[G_{\omega}\coloneqq\{1\leq i\leq N_{\omega} :  \#[\a_i]\geq 2^{p_{*} k}\textrm{ and }|r_{[\a_i]}|\geq \exp(-e^{3\alpha k/4}) \}.\] 
			Since $\omega \in \Omega_{2}$ (and therefore $\omega\in\tilde{\Omega}_1\cap \tilde{\Omega}_{3}$), and $N_{\omega}\geq N'$, we know that 
			\begin{equation}\label{e:manydecaylevels}
				\# G_{\omega}\geq N_{\omega}(1-2e^{-\alpha k}).
			\end{equation}
			We call each $i\in G_{\omega}$ a decay level. 
			We enumerate the decay levels by \[i_{1} < \cdots <i_{\# G_{\omega}}.\] 
			For each decay level $i_{l}$ we can choose two distinct words  $\a,\b \in [\a_{i_{l}}]$ and $1\leq j\leq k$ such that $a_n=b_n$ for $n\neq j$ and $a_j=(j_{*},l_{1})$ and $b_{j}=(j_{*},l_{2})$. We let $t_{1,i_{l}}$ and $t_{2,i_{l}}$ be the translation parameters for the maps $\gamma_{\a}$ and $\gamma_{\b}$, i.e. they satisfy 
			\[ \gamma_{\a}(x) = r_{[\a_{i_{l}}]}x + t_{1,i_{l}}; \qquad \gamma_{\b}(x) = r_{[\a_{i_{l}}]}x + t_{2,i_{l}}.\]
			Without loss of generality we may assume that $t_{1,i_{l}}> t_{2,i_{l}}$. 
			Recall from~\eqref{e:csep} that $c>0$ is the constant such that $|t_{l_1,j^{*}}-t_{l_{2},j^*}|\geq c$. 
			Because of our assumptions on $\a$ and $\b$, we know that
			\begin{align}
				\label{e:t distances}
				1\geq t_{1,i_{l}}- t_{2,i_{l}}=\sum_{n=1}^{k}t_{a_n}\prod_{q=1}^{n-1}r_{a_q}-\sum_{n=1}^{k}t_{b_n}\prod_{q=1}^{n-1}r_{b_q}&=\left(t_{l_1,j^{*}}-t_{l_{2},j^*}\right)\prod_{q=1}^{j-1}r_{a_{q}}\\
				&\geq c\exp(-e^{3\alpha k/4})\nonumber.
			\end{align} 
			In the final inequality we used~\eqref{e:csep} and the fact that $|r_{[\a_i]}|\geq \exp(-e^{3\alpha k/4})$ when $i$ is a decay level.
			
			Our strategy for proving what remains of this proposition is to examine for each decay level $i_{l}$ and $\xi\in [-T,T]$ the quantities \[\xi \cdot t_{1,i_{l}}\cdot \prod_{j=1}^{i_{l}-1}r_{[\a_j]}\qquad\text{ and }\qquad \xi \cdot t_{2,i_{l}}\cdot \prod_{j=1}^{i_{l}-1}r_{[\a_j]}.\] In particular, we will be interested in the distance between $\mathbb{Z}$ and the difference between these terms. If a decay level is such that the distance is large, then we can apply Lemma~\ref{l:fourierdecay} in a meaningful way to bound the expression on the right hand side of~\eqref{e:decaybound}. \bigskip
			
			\noindent \textbf{Part 2. Introducing $\Bad(\xi)$ and bounding $|\widehat{\mu_{\omega}}(\xi)|$ when $\Bad(\xi)$ is small. }
			
			With the above strategy in mind, for each $1\leq l\leq \# G_{\omega}$ and $\xi\in [-T,T]$, we let $p_{l}(\xi)\in\mathbb{Z}$ and $\epsilon_{l}(\xi)\in [-1/2,1/2)$ be such that 
			\begin{equation}\label{e:eksequence}
            \xi(t_{1,i_{l}}-t_{2,i_{l}})\prod_{j=1}^{i_{l}-1}r_{[\a_j]} = p_{l}(\xi)+\epsilon_{l}(\xi).
            \end{equation}
			We observe that 
			\begin{equation}
				\label{e:pepsrelation}
				p_{l}(\xi)+\epsilon_{l}(\xi) = \xi(t_{1,i_{l}}-t_{2,i_{l}})\prod_{j=1}^{i_{\# G_{\omega}}}r_{[\a_j]} \cdot \left(\prod_{j'=i_{l}}^{i_{\# G_{\omega}}} r_{[\a_{j'}]}\right)^{-1}.
			\end{equation}
			Equation~\eqref{e:pepsrelation} implies that for all $1\leq l<l'\leq \# G_{\omega}$ we have the following relation between $p_{l}(\xi)+\epsilon_{l}(\xi)$ and $p_{l'}(\xi)+\epsilon_{l'}(\xi)$:
			\begin{equation}\label{e:successivepeps}
				p_{l}(\xi)+\epsilon_{l}(\xi)=\frac{(t_{1,i_{l}}-t_{2,i_{l}})}{(t_{1,i_{l'}}-t_{2,i_{l'}})}\cdot \left(\prod_{j=i_{l}}^{i_{l'}-1}r_{[\a_j]}\right)^{-1}\left(p_{l'}(\xi)+\epsilon_{l'}(\xi)\right).
			\end{equation}
			Moreover, by~\eqref{e:t distances} we know that 
			\begin{equation}\label{e:pointofeta}
				c\exp(- e^{3\alpha k/4})\leq \frac{(t_{1,i_{l}}-t_{2,i_{l}})}{(t_{1,i_{l'}}-t_{2,i_{l'}})}\leq \frac{1}{c}\exp( e^{3\alpha k/4})
			\end{equation}for all $1\leq l<l' \leq \# G_{\omega}$. 
			
		If $i_{l}+1=i_{l+1}$ then $\prod_{j=i_{l}}^{i_{l+1}-1}r_{[\a_j]} = r_{[\a_{i_l}]}$ so it follows from~\eqref{e:pointofeta} and the fact $i_{l}$ is a decay level that we have the upper bound
		\begin{equation}
			\label{e:expansion}	
			\left|\frac{(t_{1,i_{l}}-t_{2,i_{l}})}{(t_{1,i_{l+1}}-t_{2,i_{l+1}})}\cdot \left(\prod_{j=i_{l}}^{i_{l+1}-1}r_{[\a_j]}\right)^{-1}\right|\leq \frac{1}{c} \exp(2e^{3\alpha k/4}).
		\end{equation}
		Let \[\epsilon^* \coloneqq \frac{c\exp(-2 e^{3\alpha k/4})}{5 }.\] 
		The number $\epsilon^*$ has the important property that if $i_{l+1}=i_{l}+1$ and $\max\{|\epsilon_{l}(\xi)| , |\epsilon_{l+1}(\xi)|\} \leq \epsilon^*$, then $p_{l}(\xi)$ is uniquely determined by $p_{l+1}(\xi)$. This follows because if one assumes $i_{l}+1=i_{l+1}$ and $|\epsilon_{l+1}(\xi)|\leq \epsilon^{*}$, then by~\eqref{e:successivepeps} and~\eqref{e:expansion}, $p_{l}(\xi) + \epsilon_{l}(\xi)$ belongs to an interval of length at most $2/5$.
		
		For each $\xi\in [-T,T]$ we consider the set \[ \Bad(\xi)\coloneqq\left\{i\in G_{\omega}:\xi(t_{1,i}-t_{2,i})\prod_{j=1}^{i-1}r_{[\a_j]}\in \mathbb{Z}+[-\epsilon^{*},\epsilon^{*}]\right\}.\] If $\xi$ is such that
		\begin{equation}\label{e:assumebadissmall}
		\#\Bad(\xi)\leq \# G_{\omega}\left(1-\frac{1}{k!}\right),
		\end{equation} then by Lemma~\ref{l:fourierdecay}\footnote{In our application of Lemma~\ref{l:fourierdecay} we are implicitly using the fact that there are at most $k$ choices of $\p$. This is the case because $\p$ is the uniform probability vector on $2^{l}$ elements for some $1\leq l\leq k$.} and~\eqref{e:decaybound},  there exists $\upsilon\in(0,1)$ depending only upon our IFS and $k$ such that 
		\[|\widehat{\mu_{\omega}}(\xi)|\leq \upsilon^{\# G_{\omega}/k!}\leq \upsilon^{N_{\omega}(1-2e^{-\alpha k})/k!}.\]
		Now using~\eqref{e:T and N comparison}, we see that~\eqref{e:assumebadissmall} implies that there exists a small constant $\epsilon>0$ (depending only on our IFS and $k$) such that $|\widehat{\mu_{\omega}}(\xi)|\leq T^{-\epsilon}$. \bigskip
		
		\noindent \textbf{Part 3. Making a choice of large $\Bad(\xi)$.}
		
		For the remainder of the proof, we suppose that $\xi\in [-T,T]$ is such that 
		\[\#\Bad(\xi)\geq \#G_{\omega}\left(1-\frac{1}{k!}\right).\] 
		Under this assumption, we proceed to determine how many different possible choices of $p_{1}(\xi)$ there are (this will be the main focus of parts 3, 4, 5 and 6 of the proof). 
		Now, $\Bad(\xi)$ is a subset of $G_{\omega}$ with cardinality at least $\#G_{\omega}(1-\frac{1}{k!})$. As such, by Stirling's formula, we can bound the number of choices for $\Bad(\xi)$ from above by 
		\begin{align}\label{e:stirling}
			\begin{split}
				\sum_{i= \lceil \# G_{\omega}(1-1/k!) \rceil}^{\# G_{\omega}} {\binom{\# G_{\omega}}{i}} &\leq \frac{\#G_{\omega}}{k!}	 {\binom{\# G_{\omega}}{\lceil \# G_{\omega}(1-1/k!) \rceil}}\\ &\preceq \frac{\#G_{\omega}}{k!}\cdot e^{2\# G_{\omega}\cdot (-\frac{1}{k!}\log \frac{1}{k!}-(1-\frac{1}{k!})\log (1-\frac{1}{k!}))} \\*
				&\leq \frac{N_{\omega}}{k!}\cdot e^{2N_{\omega}\cdot (-\frac{1}{k!}\log \frac{1}{k!}-(1-\frac{1}{k!})\log (1-\frac{1}{k!}))}\\*
				&=T^{o_{k}(1)}.
			\end{split}
		\end{align}
		In the final line we used~\eqref{e:T and N comparison} and the inequality $x\leq e^x$ for $x\geq 0$.

		Let $\I\subset G_{\omega}$ be a specific choice for $\Bad(\xi)$. Let us enumerate the elements of $\I$ by $i_{l(1)}<i_{l(2)}<\cdots <i_{l(\# \I)}$. We also let 
		\begin{equation}\label{e:definejset}
		J=\{i_{l(n)}\in \I:i_{l(n)}+1<i_{l(n+1)}\}.
		\end{equation} 
		Combining our assumptions $\# \I \geq \# G_{\omega}(1-1/k!)$ and $\# G_{\omega}\geq N_{\omega}(1-2e^{-\alpha k})$ gives  
		\begin{equation*}
			\label{e:I cardinality}
			\# \I\geq N_{\omega}(1-1/k!)(1-2e^{-\alpha k})=N_{\omega}(1-\zeta_{k}) 
		\end{equation*} 
		for some $\zeta_k$ satisfying  
		\[\zeta_{k}=\O(e^{-\alpha k}).\]
		
		\noindent \textbf{Part 4. Bounding the number of choices for $p_{l(\# \I)}(\xi)$ given our fixed $\I$.}
		
		We now derive an upper bound for the number of choices for $p_{1}(\xi)$ for those $\xi\in [-T,T]$ satisfying $\Bad(\xi)=\I$. 
		
		Let $\xi\in [-T,T]$ be such that $\Bad(\xi)=\I$. We begin by remarking that for $1\leq l<l'\leq \# G_{\omega}$, by~\eqref{e:successivepeps} and~\eqref{e:pointofeta}, given $p_{l'}(\xi)$ there are at most
		\begin{equation}\label{e:integerchoice}
			\frac{\prod_{j=i_{l}}^{i_{l'}-1}|r_{[\a_{j}]}|^{-1}}{c\exp(- e^{3\alpha k/4})}
		\end{equation}
		choices for $p_{l}(\xi)$. 
		
		From~\eqref{e:pepsrelation}, 
		\begin{equation}\label{e:smallestp}
			p_{l(\# \I)}(\xi)+\epsilon_{l(\# \I)}(\xi) = \xi(t_{1,i_{l(\# \I)}}-t_{2,i_{l(\# \I)}})\prod_{j=1}^{N_{\omega}}r_{[\a_j]}\cdot \left(\prod_{j'=i_{l(\# \I)}}^{N_{\omega}}r_{[\a_{j'}]}\right)^{-1}.
		\end{equation}
		Using~\eqref{e:asymptotic inclusion}, Lemma~\ref{l:little o lemma}, and the fact $|\xi|\leq T$, we know that $\xi \prod_{j=1}^{N_{\omega}}r_{[\a_j]}$ belongs to an interval of length $\preceq_{k}T^{o_{k}(1)}$. 
		Combining this observation with
		\eqref{e:t distances}, we see that~\eqref{e:smallestp} implies that $p_{l(\# \I)}(\xi)$ belongs to an interval of size \[\preceq_{k}T^{o_{k}(1)}\cdot 	\left|\left(\prod_{j=i_{l(\# \I)}}^{N_{\omega}}r_{[\a_j]}\right)^{-1}\right|\]
		Therefore the number of choices for $p_{l(\# \I)}(\xi)$ is bounded above by 
		\begin{equation}\label{e:firstintegerchoice}
			\preceq_{k}T^{o_{k}(1)}\cdot 	\left|\left(\prod_{j=i_{l(\# \I)}}^{N_{\omega}}r_{[\a_j]}\right)^{-1}\right|.
		\end{equation}
		Since $\I\subset \{1,\ldots, N_{\omega}\}$ and  $\# \I\geq N_{\omega}(1-\zeta_{k})$, the product 
		\[\left|\left(\prod_{j=i_{l(\# \I)}}^{N_{\omega}}r_{[\a_j]}\right)^{-1}\right|\] contains at most $N_{\omega}\zeta_{k}$ terms. If a term in this product satisfies $|r_{[\a_j]}|\geq \exp(- e^{3\alpha k/4})$ then we bound $|r_{[\a_j]}|^{-1}$ from above by $\exp( e^{3\alpha k/4})$. 
		We collect the remaining terms which satisfy $|r_{[\a_j]}|< \exp(- e^{3\alpha k/4})$ and bound their contribution to this product from above by  
		\[\prod_{\substack{1 \leq i \leq N_{\omega} \\ r_{[\a_i]} < \exp(- e^{3\alpha k/4})}} |r_{[\a_i]}|^{-1}.\] 		
		Combining these bounds, we see that 
		\[\left|\left(\prod_{j=i_{l(\# \I)}}^{N_{\omega}}r_{[\a_j]}\right)^{-1}\right|\leq \exp( e^{3\alpha k/4}N_{\omega}\zeta_{k})\cdot \prod_{\substack{1 \leq i \leq N_{\omega} \\ r_{[\a_i]} < \exp(- e^{3\alpha k/4})}} |r_{[\a_i]}|^{-1}.\] Now using the fact that $\zeta_{k}=\O(e^{-\alpha k})$, together with Lemma~\ref{l:little o lemma} and the fact $\omega\in \Omega_{2}$ and therefore $\omega\in \tilde{\Omega}_{4}$, we see that 
		\[\left|\left(\prod_{j=i_{l(\# \I)}}^{N_{\omega}}r_{[\a_j]}\right)^{-1}\right|\leq (e^{N_{\omega}})^{o_{k}(1)}=T^{o_{k}(1)},\] 
		where in the final inequality we used~\eqref{e:T and N comparison}. 
		Therefore by~\eqref{e:firstintegerchoice} we have $\preceq_{k} T^{o_{k}(1)}$ choices for $p_{l(\# \I)}(\xi)$.\bigskip 
	
	\noindent \textbf{Part 5. Bounding the number of choices for $p_{1}(\xi)$ given $p_{l(\#\I)}(\xi)$.}
	
	Now suppose $p_{l(\#\I)}(\xi)$ is given. Recall the set $J$ from~\eqref{e:definejset}. 
	We can use the bound provided by~\eqref{e:integerchoice} repeatedly to the elements of $J$, and use the fact that our choice of $\epsilon^*$ means that if ${i}_{l}+1=i_{l+1}$ then $p_{l+1}(\xi)$ uniquely determines $p_{l}(\xi)$, to assert that we have at most 
	\begin{equation}
		\label{e:integercount3}
		\prod_{i_{l(n)}\in J}\prod_{j=i_{l(n)}}^{i_{l(n+1)}-1}\frac{|r_{[\a_j]}|^{-1}}{c\exp (- e^{3\alpha k/4})}
	\end{equation}
	choices for $p_{l(1)}(\xi)$. We now argue as in Part~4 above. Since $\# \I\geq N_{\omega}(1-\zeta_k)$, the above product contains at most $2N_{\omega}\zeta_{k}$ terms. 
	For each term in this product, if $|r_{[\a_j]}|\geq \exp(-e^{3\alpha k/4})$ then we bound $|r_{[\a_j]}|^{-1}$ from above by $\exp( e^{3\alpha k/4})$. We collect all of the remaining terms which satisfy $|r_{[\a_j]}|< \exp(- e^{3\alpha k/4})$ and bound their contribution to this expression from above by 
	\[ \prod_{\substack{1 \leq i \leq N_{\omega} \\ r_{[\a_i]} < \exp(- e^{3\alpha k/4})}} |r_{[\a_i]}|^{-1}.\] 
	Applying both of these bounds, and using the fact this product contains at most $2N_{\omega}\zeta_{k}$ terms, gives that there are at most 
	\begin{equation}\label{e:pl1count}
	c^{-2N_\omega\zeta_k}\cdot \exp(4 e^{3\alpha k/4}N_{\omega}\zeta_{k})\cdot\prod_{\substack{1 \leq i \leq N_{\omega} \\ r_{[\a_i]} < \exp(- e^{3\alpha k/4})}} |r_{[\a_i]}|^{-1}
	\end{equation}
	choices for $p_{l(1)}(\xi)$. 
	Now by an analogous argument to that given above we can show that the quantity in~\eqref{e:pl1count} is $\preceq_{k} T^{o_{k}(1)}$, so given $p_{l(\#I)}(\xi)$ we have $\preceq_{k} T^{o_{k}(1)}$ choices for $p_{l(1)}(\xi)$. 
	
	Finally, given $p_{l(1)}(\xi)$, we can apply~\eqref{e:integerchoice} to bound the number of choices for $p_{1}$ from above by 
	\[\frac{\prod_{j=i_{1}}^{i_{l(1)} - 1}|r_{[\a_j]}|^{-1}}{c\exp (- e^{3\alpha k/4})}.\]
	Using the fact that this product contains at most $N_{\omega}\zeta_{k}$ terms we can proceed by an analogous argument to that given above to show that 
	\[\frac{\prod_{j=i_{1}}^{i_{l(1)} - 1}|r_{[\a_j]}|^{-1}}{c\exp (- e^{3\alpha k/4})}\preceq_{k}T^{o_{k}(1)}.\]
	Therefore given $p_{l(1)}(\xi)$ we have $\preceq_{k} T^{o_{k}(1)}$ choices for $p_{1}(\xi)$.\bigskip
	
	\noindent \textbf{Part 6. Collecting our counting bounds and concluding.}
	
	Combining the counting bounds obtained above, we see that for a specific choice of $\I$ we have at most $C_{k,1}T^{o_{k}(1)}$ choices for $p_{1}(\xi)$, where $C_{k,1}>0$ is a constant depending upon $k$.
	Combining this with our bound~\eqref{e:stirling} coming from Stirling's inequality, we have a total (across all possible choices of $\I$) of at most $C_{k,2}T^{o_{k}(1)}$ 
	choices for $p_{1}(\xi)$, where $C_{k,2}$ is some other constant depending upon $k$.
	
	Recall now the identity 
	\[ \xi(t_{1,i_1}-t_{2,i_1})\prod_{j=1}^{i_{1}-1}r_{[\a_j]}=p_{1}(\xi)+\epsilon_{1}(\xi).\] 
	Rearranging this expression, we see that if $\xi$ is a frequency with $|\widehat{\mu_{\omega}}(\xi)| > T^{-\epsilon}$ for the choice of $\epsilon$ given above, then it belongs to at most $C_{k,2}T^{o_{k}(1)}$ intervals of length 
	\[\left|\left((t_{1,i_1}-t_{2,i_1})\prod_{j=1}^{i_{1}-1}r_{[\a_j]}\right)^{-1}\right|.\] 
	This product contains at most $N_{\omega}\zeta_{k}$ terms. 
	Thus if we replicate the arguments given above (recall Part~4, for example), we can show that 
	\[\left|\left((t_{1,i_1}-t_{2,i_1})\prod_{j=1}^{i_{1}-1}r_{[\a_j]}\right)^{-1}\right|\preceq_{k} T^{o_{k}(1)}.\] 
	Proposition~\ref{p:decayoutsidesparse} now follows once we observe that any interval of length $T^{o_{k}(1)}$ can be covered by at most $T^{o_{k}(1)}$ intervals of length $1$. 
\end{proof}

Before proceeding with the proof of Theorem~\ref{t:main}, we give a consequence of Proposition~\ref{p:decayoutsidesparse} which may be of interest in its own right. 
In~\cite{Kau,Tsujii,MS} it has been proved that the set of frequencies for which the Fourier transform of a self-similar measure in the line does not exhibit polynomial decay is `sparse' in a precise sense. 
In~\cite[Corollary~1.8]{KhalilPreprintfourier}, such estimates have also been proved for a wide class of measures which are not necessarily self-similar. 
In Corollary~\ref{c:tsujii}, we generalise the main result of Tsujii~\cite{Tsujii} to stationary measures for a class of CIFSs consisting of similarities; these measures do not generally satisfy the uniform affine non-concentration condition assumed by Khalil~\cite[Corollary~1.8]{KhalilPreprintfourier}. 
\begin{cor}\label{c:tsujii}
Let $\p$ be a probability vector and $\Phi$ be a non-trivial CIFS acting on $\R$ consisting of similarities. Assume that $\sum_{a\in \A}p_a|r_a|^{-\tau}<\infty$ for some $\tau>0$. Let $\mu$ be the stationary measure for $\Phi$ and $\p$. 
Then for all $\gamma > 0$ there exist $\eta,C>0$ such that for all $T>0$, the set $\{ \xi \in [-T,T] : |\widehat{\mu}(\xi) | \geq T^{-\eta} \}$ can be covered by at most $C T^{\gamma}$ intervals of length~$1$. 
\end{cor}
\begin{proof}
Assume for the purposes of a contradiction that this result is false. 
The idea will be to use this to find well-separated frequencies where $|\widehat{\mu}|$ is large, but show that Proposition~\ref{p:decayoutsidesparse} implies that the mean average of $|\widehat{\mu}|$ over these frequencies must be small, which will give a contradiction. We use this averaging argument to overcome the fact that in Proposition~\ref{p:decayoutsidesparse}, different $\omega$ may give rise to different bad frequencies. 

Firstly, there exists $\gamma > 0$ such that for all $\eta > 0$ there exists a sequence of positive numbers $T_n \to \infty$ and frequencies 
\begin{equation}\label{e:tsujiifreq} 
\{\xi_{n,m}\}_{n \geq 1, 1 \leq m \leq \lceil T_n^{\gamma} \rceil } \in [-T_n,T_n] 
\end{equation}
such that for all $n \in \N$ we have $|\widehat{\mu}(\xi_{n,m})| \geq T_n^{-\eta}$ for all $m$, and $\xi_{n,m+1} > \xi_{n,m} + 1$ for all $m < \lceil T_n^{\gamma} \rceil - 1$. %
Fix $k \in \N$ large enough that the $o_k(1)$ term in Proposition~\ref{p:decayoutsidesparse} is smaller than $\gamma / 4$. 
Then there exists $\epsilon > 0$ such that (setting $T'=T$) for all $T>0$ sufficiently large there exists $\Omega_T \subset \Omega$ such that $P(\Omega \setminus \Omega_T) \leq T^{-\epsilon}$, and such that for all $\omega \in \Omega_T$, the set $\{\xi \in [-T,T] : |\widehat{\mu_{\omega}}(\xi)|\geq T^{-\epsilon} \}$ can be covered by at most $T^{\gamma/4}$ intervals of length~$1$. 
The previous sentence is also clearly true if $\epsilon$ is replaced by any $\epsilon' \leq \epsilon$, and in particular for $\epsilon' = \min\{\epsilon,\gamma / 20\}$. 
Henceforth fix $\eta \coloneqq \min\{\epsilon/4,\gamma/80\}$, and let $\{\xi_{n,m}\}$ be the set of frequencies from~\eqref{e:tsujiifreq} corresponding to this value of $\eta$. 

By our disintegration from Corollary~\ref{c:disintegrateinreals}, for all $n$ sufficiently large and $1 \leq m \leq \lceil T_n^{\gamma} \rceil$, 
\[ |\widehat{\mu}(\xi_{n,m})| \leq \int_{\Omega} |\widehat{\mu_{\omega}}(\xi_{n,m})| dP(\omega) \leq P(\Omega \setminus \Omega_{T_n}) + \int_{\Omega_{T_n}} |\widehat{\mu_{\omega}}(\xi_{n,m})| dP(\omega) , \]
and rearranging gives 
\begin{equation}\label{e:tsujiiint} 
\int_{\Omega_{T_n}} |\widehat{\mu_{\omega}}(\xi_{n,m})| dP(\omega) \geq |\widehat{\mu}(\xi_{n,m})| -P(\Omega \setminus \Omega_{T_n}) \geq T_n^{-\eta} - T_n^{-4\eta} \geq T_n^{-2\eta}. 
\end{equation}

Observe that for all $n$ sufficiently large and all $\omega \in \Omega_{T_n}$, since the frequencies are separated by gaps larger than $1$, our application of Proposition~\ref{p:decayoutsidesparse} above gives that 
\[ \# \{ 1 \leq m \leq \lceil T_n^{\gamma} \rceil : |\widehat{\mu_{\omega}}(\xi_{n,m})|\geq T_n^{-4\eta} \} \leq T_n^{\gamma / 4}, \]
so 
\begin{equation}\label{e:tsujiiav} 
(\lceil T_n^{\gamma} \rceil)^{-1} \sum_{m = 1}^{\lceil T_n^{\gamma} \rceil} |\widehat{\mu_{\omega}}(\xi_{n,m})| \leq (\lceil T_n^{\gamma} \rceil)^{-1} \cdot T_n^{\gamma / 4} +  T_n^{-4\eta} \leq T_n^{-3\eta}.
\end{equation}
Combining~\eqref{e:tsujiiint} and~\eqref{e:tsujiiav} gives that for all $n$ sufficiently large,  
\begin{align*} 
T_n^{-2\eta} &\leq (\lceil T_n^{\gamma} \rceil)^{-1} \sum_{m = 1}^{\lceil T_n^{\gamma} \rceil} \int_{\Omega_{T_n}} |\widehat{\mu_{\omega}}(\xi_{n,m})| dP(\omega) \\
&= \int_{\Omega_{T_n}} (\lceil T_n^{\gamma} \rceil)^{-1} \sum_{m = 1}^{\lceil T_n^{\gamma} \rceil} |\widehat{\mu_{\omega}}(\xi_{n,m})| dP(\omega) \\
&\leq T_n^{-3\eta}, 
\end{align*}
which is a contradiction. This completes the proof. 
\end{proof}

\subsection{Decay of non-linear images} 

We now show that for a $P$-large set of $\omega$, the magnitude of the Fourier transform of $F \tilde{\mu}_{\omega}$ can be bounded above by some particular power of the frequency multiplied by some constant depending upon the first and second derivative of $F$. 
Given a $C^{2}$ function $F\colon [0,1]^{d+1}\to \mathbb{R}$, we recall the following notation:
\[\|F\|_{\infty,1}\coloneqq \max_{x\in [0,1]^{d+1}}\left|\frac{\partial F}{\partial x_{d+1}}(x)\right|,\quad\|F\|_{\infty,2}\coloneqq \max_{x\in [0,1]^{d+1}}\left|\frac{\partial^{2}F}{\partial x_{d+1}^2}(x)\right|\] and
\[\|F\|_{\min,2}\coloneqq \min_{x\in [0,1]^{d+1}}\left|\frac{\partial^{2}F}{\partial x_{d+1}^2}(x)\right|.\]
Recall also that $s_{\Phi}$ is the Frostman exponent from Proposition~\ref{p:frostman}. 
Our proof of the following proposition uses ideas of Mosquera and Shmerkin~\cite{MS}, with additional complications related to how to choose $\omega$ and the possible presence of infinitely many maps. 
\begin{prop}\label{p:fmuomegadecay}
For all $k\in\mathbb{N}$ sufficiently large there exist $C,\beta,\beta',\epsilon>0$ such that for all $C^{2}$ functions $F\colon [0,1]^{d+1}\to\mathbb{R}$ satisfying $\frac{\partial^{2}F}{\partial x_{d+1}^2}(x) \neq 0$ for all $x\in[0,1]^{d+1}$, for all $\xi\neq 0$ there exists $\Omega_{\xi} \subseteq \Omega$ satisfying \[P(\Omega \setminus \Omega_{\xi}) \leq C (1+\|F\|_{\infty,1}^{-\epsilon})(1+\|F\|_{\min,2}^{-\beta})|\xi|^{-\beta'},\] 
and such that for all $\omega \in \Omega_{\xi}$ we have 
\begin{align*}|\widehat{F \tilde{\mu}}_{\omega}(\xi)| \leq C& (1+\|F\|_{\infty,1}+\|F\|_{\infty,1}^{-\epsilon}+\|F\|_{\infty,2})\\*
	& \times (1+\|F\|_{\min,2}^{-s_{\Phi}})|\xi|^{-\min\{ s_{\Phi}/3 -o_{k}(1), \epsilon/3,0.2 \}}.
	\end{align*}
\end{prop}%
\begin{proof}
Without loss of generality we assume $\frac{\partial^2 F}{\partial x_{d+1}^2}(x) > 0$ for all $x\in [0,1]^{d+1}$. 
Let $\alpha$ be as in Proposition~\ref{p:largedev}. 
Fix $k \in \N$ sufficiently large so that Propositions~\ref{p:largedev},~\ref{p:frostman} and~\ref{p:decayoutsidesparse} hold. Fix $\varepsilon$ and $C_k$ as given by Proposition~\ref{p:decayoutsidesparse} for this choice of $k$. 
We will only consider positive $\xi$ (the proof when $\xi < 0$ is analogous), and in fact we may assume $\xi > 1$ since the case $\xi \in [0,1]$ can be dealt with by increasing $C$. 
We let $N'\in \N$ be such that $e^{2\Lambda kN'}<\xi^{2/3}\leq e^{2\Lambda k(N'+1)}$. Let $\Omega^{*}$ be as in Proposition~\ref{p:largedev} for these values of $k$, $\alpha$, $N'$, so by this proposition there exists $\beta_{k}>0$ such that 
\begin{equation}
	\label{e:firstomegabound}
	1-P(\Omega^{*})\preceq_{k} e^{-\beta_{k}N'}\preceq_{k} \xi^{-\beta_{k}/(3\Lambda k)},
\end{equation} 
where in the final inequality we used that $\xi^{2/3}\leq e^{2\Lambda k(N'+1)}$.

Given $\omega \in \Omega$, we let $N_{\omega} \in \N$ be such that 
\[ \prod_{i=1}^{N_{\omega}}| r_{[\a_i]}| \geq \xi^{-2/3} > \prod_{i=1}^{N_{\omega}+1}| r_{[\a_i]}|. \] 
Appealing to the definition of $\Omega^{*}$ and $N'$, we see that there exists $c'>1$ such that if $\omega\in \Omega^{*}$ then
\begin{equation}
	\label{e:Ncomparisons2}
	N'\leq N_{\omega}\leq c'N'.
\end{equation}
 Again appealing to the definition of $\Omega^{*}$, we see that 
\[ |r_{[\a_{N_{\omega}+1}]}|\geq \min\left\{\exp(-e^{3\alpha k/4}),\exp\left(2N_{\omega}\sum_{\stackrel{[\a]\in \I}{|r_{[\a]}|<\exp(-e^{3\alpha k/4})}}q_{[\a]}\log|r_{[\a]}|\right)\right\}.\] 
As such, if we apply Lemma~\ref{l:little o lemma} together with the inequalities $e^{2\Lambda kN'}<\xi^{2/3}\leq e^{2\Lambda k(N'+1)}$ and~\eqref{e:Ncomparisons2}, we can conclude that if $\omega\in \Omega^{*}$ then 
\[\prod_{i=1}^{N_{\omega}}| r_{[\a_i]}|\leq \xi^{-2/3}|r_{[\a_{N_{\omega}+1}]}|^{-1}\preceq_{k} \xi^{-2/3+o_{k}(1)}.\] 
We can therefore assume that $k$ has been chosen sufficiently large that there exists $C_{k}' > 0$ such that 
\begin{equation}
	\label{e:xi contraction}
	\xi^{-2/3}\leq \prod_{i=1}^{N_{\omega}}| r_{[\a_i]}|\leq C_{k}' \xi^{-0.6}
\end{equation}
for all $\omega\in \Omega^{*}$.

Recall the definition of $\mu_{\omega}$ from~\eqref{e:convstructure}, and recall that $S_{\lambda}(x) = \lambda x$. We see that $\mu_{\omega} = \mu_{N_{\omega}} \ast \lambda_{N_{\omega}}$, where 
\[ \mu_{N_{\omega}} \coloneqq \ast_{i=1}^{N_{\omega}}\frac{1}{\# [\a_i]}\sum_{\b\in [\a_i]}\delta_{t_{\b,[\a_i]}\cdot \prod_{j=1}^{i-1}r_{[\a_j]}}; \qquad \lambda_{N_{\omega}} \coloneqq S_{\prod_{i=1}^{N_{\omega}} r_{[\a_i]}} \mu_{\sigma^{N_{\omega}} \omega}.\] 
For each $\omega\in \Omega$ let $F_{\omega} \colon \mathbb{R}\to\R$ be the function given by $F_{\omega}(x)=F(\x_{\omega},x)$ (recall \eqref{e:x_omega def} for the definition of $\x_{\omega}$).
By considering the Taylor expansion of each $F_{\omega}$ and using a similar calculation to the one given in the proof of \cite[Theorem~3.1]{MS}, we obtain the following for all $\omega\in \Omega^{*}$: 
\begin{align}\label{e:longfouriercalc}
	\begin{split}
		&\phantom{--}|\widehat{F \tilde{\mu}}_{\omega}(\xi)|\\
	&= \Big| \int_{[0,1]^{d+1}} e(\xi F(x)) d\tilde{\mu}_{\omega}(x) \Big| \qquad \\
		&=\Big| \int_{[0,1]} e(\xi F_{\omega}(y)) d\mu_{\omega}(y) \Big| \\
		&= \Big| \int_{[0,1]}  \int_{[0,1]}  e(\xi F_{\omega}(y_1+y_2)) d\mu_{N_{\omega}}(y_1) d\lambda_{N_{\omega}}(y_2) \Big| \\
		&= \Big| \int_{[0,1]}  \int_{[0,1]}  e\left(\xi F_{\omega}(y_1) + \xi F_{\omega}'(y_1) y_2\right)(1+\O(\|F\|_{\infty,2}\xi y_{2}^2))d\mu_{N_{\omega}}(y_1) d\lambda_{N_{\omega}}(y_2) \Big| \\
		&\leq \Big| \int_{[0,1]} e(\xi F_{\omega}(y_1)) \left( \int_{[0,1]} e\left(\xi F_{\omega}'(y_1) y_{2}\right) d\lambda_{N_{\omega}}(y_2) \right) d\mu_{N_{\omega}}(y_1) \Big|\\
		&\phantom{--} + \O_{k}\left(\|F\|_{\infty,2}\xi^{-0.2}\right) \\
		&\leq \int_{[0,1]} | \widehat{\lambda_{N_{\omega}}}\left(\xi F_{\omega}'(y_1) \right) | d\mu_{N_{\omega}}(y_1) + \O_{k}\left(\|F\|_{\infty,2}\xi^{-0.2}\right) \\
		&= \int_{[0,1]} \Big| \widehat{\mu_{\sigma^{N_{\omega}} \omega}}\left(\xi F_{\omega}'(y_1)  \prod_{i=1}^{N_{\omega}} r_{[\a_i]}\right) \Big| d\mu_{N_{\omega}}(y_1) + \O_{k}\left(\|F\|_{\infty,2}\xi^{-0.2}\right). 
	\end{split}
\end{align}
We emphasise that the constant implicit in the $\mathcal{O}$ and $\mathcal{O}_k$ notation does not depend on $F$, and note that we used~\eqref{e:xi contraction} for the first inequality. 

Let $T' = \|F\|_{\infty,1} \xi^{1/3} $ and $T \coloneqq \|F\|_{\infty,1} \cdot \xi \cdot \prod_{i=1}^{N_{\omega}}| r_{[\a_i]}|$. Then by~\eqref{e:xi contraction}, we have
\[T' \leq T \leq \|F\|_{\infty,1}C_{k}'\xi^{0.4}\] for all $\omega\in \Omega^{*}$. 
We now apply Proposition~\ref{p:decayoutsidesparse} for our previously fixed choice of $k$ (which defined $\varepsilon$ and $C_k$) and this value of $T'$. We let $\Omega_{2}$  be as in the statement of this proposition. 
Let us now fix $\omega\in \Omega^{*}$ satisfying $\sigma^{N_{\omega}} \omega \in \Omega_{2}$. 
By Proposition~\ref{p:decayoutsidesparse} let $I_1,\dotsc,I_{\lfloor C_{k} T^{o_{k}(1)} \rfloor}$ be intervals of length~$1$ covering the set of frequencies $\xi' \in [-T,T]$ for which $|\widehat{\mu_{\sigma^{N_{\omega}}\omega}}(\xi')| \geq T^{-\epsilon}$. 
Define  
\begin{equation*}
\Gamma \coloneqq \left\{ y_1 : \xi F_{\omega}'(y_1) \prod_{i=1}^{N_{\omega}} r_{[\a_i]} \in \bigcup_{i=1}^{\lfloor C_{k}T^{o_{k}(1)} \rfloor} I_i \right\}. 
\end{equation*}
Note that $\Gamma = \cup_{i=1}^{\lfloor C_{k}T^{o_{k}(1)} \rfloor} (F_{\omega}')^{-1} J_i$ for some intervals $J_i$ of length \[ |J_i| = \left|\left( \xi \prod_{n=1}^{N_{\omega}} r_{[\a_n]} \right)^{-1}\right| \leq \xi^{-1/3}. \] 

This is the point in the proof where we use in a crucial way our assumption that $\frac{\partial^2 F}{\partial x_{d+1}^2}(x) > 0$ for all $x\in [0,1]^{d+1}$. 
Since $\|F\|_{\min,2}>0$, we see that $\Gamma$ can be covered by $\lfloor C_{k}T^{o_{k}(1)} \rfloor$ intervals $J_i'$ of length at most $\|F\|_{\min,2}^{-1}\xi^{-1/3}$. 
Let $r' = \|F\|_{\min,2}^{-1}\xi^{-1/3} + C_{k}'\xi^{-0.6}$. We apply Proposition~\ref{p:frostman} for this choice of $r'$ and let $\Omega_{1}$ be as in the statement of this proposition. Let us now also assume that our $\omega$ satisfies $\omega\in \Omega_{1}$. 
Since $\mu_{\omega} = \mu_{N_{\omega}} * \lambda_{N_{\omega}}$, and $\mbox{supp}(\lambda_{N_{\omega}})$ is contained in an interval $[- C_{k}'\xi^{-0.6}, C_{k}'\xi^{-0.6}]$, we see that for any interval $I = (c_1,c_2)$, 
\[ \mu_{N_{\omega}} (I) \leq \mu_{\omega}(c_1 -  C_{k}'\xi^{-0.6}, c_2 +  C_{k}'\xi^{-0.6}). \]
Recalling that $s_{\Phi}$ is the Frostman exponent from Proposition~\ref{p:frostman}, it follows from this proposition, and the fact that $\omega\in\Omega_{1}$, that 
\[\mu_{N_{\omega}}(J_i') \preceq_{k}\|F\|_{\min,2}^{-s_{\Phi}} \xi^{-s_{\Phi}/3}+\xi^{-0.6 s_{\Phi}} \leq (1+\|F\|_{\min,2}^{-s_{\Phi}})\xi^{-s_{\Phi}/3}\] for all $i$. 
Therefore 
\begin{align*} 
\mu_{N_{\omega}}(\Gamma) &\preceq_{k} T^{o_{k}(1)} (1+\|F\|_{\min,2}^{-s_{\Phi}})\xi^{-s_{\Phi}/3}\\
	& \preceq_{k} (1+\|F\|_{\infty,1}) (1+\|F\|_{\min,2}^{-s_{\Phi}})\xi^{-s_{\Phi}/3+o_{k}(1)}, 
	\end{align*}
where in the final inequality we have used that $T\leq \|F\|_{\infty,1}C_{k}'\xi^{0.4}$. 

Let 
\[ \Omega_{\xi} = \{ \omega \in \Omega^{*}\cap \Omega_{1} : \sigma^{N_{\omega}} \omega \in \Omega_{2} \}. \] 
Summarising the above, from~\eqref{e:longfouriercalc} we have shown that if $\omega\in \Omega_{\xi}$ then 
\begin{align*} 
	| \widehat{F\tilde{\mu}}_{\omega}(\xi)| &\preceq_{k} \int_{\Gamma \cup (\R \setminus \Gamma)} \Big| \widehat{ \mu_{\sigma^{N_{\omega}}\omega}}\left(\xi F_{\omega}'(y_1) \prod_{n=1}^{N_{\omega}} r_{[\a_n]}\right) \Big| d\mu_{N_{\omega}}(y_1) + \|F\|_{\infty,2}\xi^{-0.2} \\
	&\preceq_{k} \mu_{N_{\omega}}(\Gamma) + T^{-\epsilon} + \|F\|_{\infty,2}\xi^{-0.2}\\
	&\preceq_{k} (1+\|F\|_{\infty,1}) (1+\|F\|_{\min,2}^{-s_{\Phi}})\xi^{-s_{\Phi}/3+o_{k}(1)}\\*
	&\phantom{--}  + \|F\|_{\infty,1}^{-\epsilon}\xi^{-\epsilon /3} + \|F\|_{\infty,2}\xi^{-0.2} \\
	&\preceq_{k} (1+\|F\|_{\infty,1}+\|F\|_{\infty,1}^{-\epsilon}+\|F\|_{\infty,2})\\*
	&\phantom{--}\times (1+\|F\|_{\min,2}^{-s_{\Phi}})\xi^{ - \min\{ s_{\Phi}/3 -o_{k}(1), \epsilon /3,0.2 \}},
\end{align*}
where in the penultimate line we used that $\|F\|_{\infty,1}\xi^{1/3}\leq T$. 

All that remains is to bound $P(\Omega \setminus \Omega_{\xi})$. 
With this goal in mind, recall the constants $\beta_k$, $\beta$ and $\epsilon$ (which may each depend on $k$) from Propositions~\ref{p:largedev}, \ref{p:frostman} and~\ref{p:decayoutsidesparse} respectively, and observe the following: 
\begin{align*}
	&P(\Omega \setminus \Omega_{\xi})\\
	&\leq P(\{\omega:\omega\notin \Omega^{*}\})+ P(\{\omega:\omega\notin \Omega_{1}\})+P(\{\omega:\sigma^{N_\omega}\omega\notin \Omega_{2}\})\\	
	&\preceq_{k} \xi^{-\beta_{k}/3\Lambda k}+ (r')^{\beta}+\sum_{k=0}^{\infty}P(\{\omega:N_{\omega}=k,\, \sigma^{k}\omega\notin \Omega_{2}\}) \qquad \text{(\eqref{e:firstomegabound} and Prop.~\ref{p:frostman})}\\
	&\preceq_{k} \xi^{-\beta_{k}/3\Lambda k}+ (r')^{\beta}+\sum_{k=0}^{\infty}P(\{\omega:N_{\omega}=k\})P(\{\omega: \sigma^{k}\omega\notin \Omega_{2}\})\quad \text{(independence)}\\
	&\preceq_{k}\xi^{-\beta_{k}/3\Lambda k}+ (r')^{\beta} + P(\Omega \setminus \Omega_2) \sum_{k=0}^{\infty}P(\{\omega:N_{\omega}=k\})\qquad \text{($\sigma$-invariance of $P$)}\\
	&=\xi^{-\beta_{k}/3\Lambda k}+ (r')^{\beta} + P(\Omega \setminus \Omega_2) \quad \quad  \left(\textrm{since }\sum_{k=0}^{\infty}P(\{\omega:N_{\omega}=k\})=1\right)\\
	&\preceq_{k} \xi^{-\beta_{k}/3\Lambda k}+(r')^{\beta}+(T')^{-\epsilon}\qquad \qquad \qquad \text{(Prop. ~\ref{p:decayoutsidesparse})}\\
	&\preceq_{k} \xi^{-\beta_{k}/3\Lambda k}+(1+\|F\|_{\min,2}^{-\beta})\xi^{-\beta/3}+\|F\|_{\infty,1}^{-\epsilon}\xi^{-\epsilon/3}\\
	&\preceq_{k}(1+\|F\|_{\infty,1}^{-\epsilon})(1+\|F\|_{\min,2}^{-\beta})|\xi|^{-\min\{\beta_{k}/3\Lambda k,\beta/3, \epsilon/3 \}}. 
\end{align*} 
In the penultimate line we used that 
\[ r' = \|F\|_{\min,2}^{-1}\xi^{-1/3} + C_{k}'\xi^{-0.6} \preceq_k (1+\|F\|_{\min,2}^{-1}) \xi^{-1/3} \] and $T'=\|F\|_{\infty,1}\xi^{1/3}$. Taking $\beta' \coloneqq \min\{\beta_{k}/3\Lambda k,\beta/3, \epsilon/3 \}$ completes the proof. 
\end{proof}

In the next proposition we will only assume that $\frac{\partial^{2}F}{\partial x_{d+1}^2}(x) \neq 0$ on the support of $\mu$. 
\begin{prop}\label{p:onlyinsupport}
For all $k\in\mathbb{N}$ sufficiently large, letting $\beta',\epsilon$ be the constants from Proposition~\ref{p:fmuomegadecay}, the following holds. 
For all $C^{2}$ functions $F\colon [0,1]^{d+1}\to\mathbb{R}$ satisfying $\frac{\partial^{2}F}{\partial x_{d+1}^2}(x) \neq 0$ for all $x\in \supp(\mu)$, there exists $C_F > 0$ (depending on $F, \mu, k$) such that for all $\xi\neq 0$ there exists $\Omega_{\xi} \subseteq \Omega$ satisfying $P(\Omega \setminus \Omega_{\xi}) \leq C_F |\xi|^{-\beta'}$, 
and such that for all $\omega \in \Omega_{\xi}$ we have 
\[ |\widehat{F \tilde{\mu}}_{\omega}(\xi)| \leq C_F |\xi|^{ - \min\{ s_{\Phi}/3 -o_{k}(1), \epsilon/3,0.2 \}} . \]
\end{prop}
\begin{proof}
Since $\supp (\mu)$ is compact, there exists $c_F > 0$ depending only on $F$ such that $\Big|\frac{\partial^{2}F}{\partial x_{d+1}^2}(x)\Big| > 2c_F$ for all $x \in \supp (\mu)$. 
Since $F''$ is uniformly continuous on $[0,1]^d$, there exists $r_F \in (0,1)$ depending only on $F$ such that if the distance from a point $x \in [0,1]^d$ to $\supp(\mu)$ is at most $r_F$ then $\Big|\frac{\partial^{2}F}{\partial x_{d+1}^2}(x)\Big| > c_F$. 

As in the proof of Proposition~\ref{p:fmuomegadecay}, fix $k \in \N$ sufficiently large so that Propositions~\ref{p:largedev},~\ref{p:frostman} and~\ref{p:decayoutsidesparse} hold, and we may assume that $\xi > 0$. 
We now follow the proof of Proposition~\ref{p:fmuomegadecay} using the same notation, to get a set $\Gamma = \cup_{i=1}^{\lfloor C_{k}T^{o_{k}(1)} \rfloor} (F_{\omega}')^{-1} J_i$ for some intervals $J_i$ of length $|J_i| \leq \xi^{-1/3}$. 
Fix any one of the intervals $J_i$, and let $A$ be the $r_F$-neighbourhood of $\supp(\mu_{\omega})$, so $\mu_{\omega}((F_{\omega}')^{-1} J_i) = \mu_{\omega}((F_{\omega}')^{-1} J_i \cap A)$. 
But the sign of $F_{\omega}''(x)$ is constant on each connected component of $A$, and $A$ has at most $\lceil r_F^{-1} \rceil$ connected components, so $(F_{\omega}')^{-1} J_i \cap A$ consists of at most $\lceil r_F^{-1} \rceil$ intervals, each of length at most $\xi^{-1/3}/c_F$. 
Assume $\omega \in \Omega_1$, where $\Omega_1$ is the set from Proposition~\ref{p:frostman} for $r' = \xi^{-1/3}/c_F  + C_k' \xi^{-0.6}$, and write $\mu_{\omega} = \mu_{N_{\omega}} * \lambda_{N_{\omega}}$. 
Then since $\mbox{supp}(\lambda_{N_{\omega}})$ is contained in an interval $[- C_{k}'\xi^{-0.6}, C_{k}'\xi^{-0.6}]$, 
\begin{align*}
\mu_{N_\omega}(\Gamma) &\preceq_k T^{o_{k}(1)} \lceil r_F^{-1} \rceil (\xi^{-1/3}/c_F + \xi^{-0.6})^{s_{\Phi}} \\
&\preceq_{k} (1+\|F\|_{\infty,1}) \lceil r_F^{-1} \rceil (1+c_F^{-s_{\Phi}})\xi^{-s_{\Phi}/3+o_{k}(1)}.
\end{align*}
The rest of the proof proceeds as for Proposition~\ref{p:fmuomegadecay}, with $(1+\|F\|_{\min,2}^{-s_{\Phi}})$ replaced by $\lceil r_F^{-1} \rceil (1+c_F^{-s_{\Phi}})$. 
\end{proof}

It is now straightforward to deduce Theorem~\ref{t:main} from Propositions~\ref{p:fmuomegadecay} and~\ref{p:onlyinsupport}. 
\begin{proof}[Proof of Theorem~\ref{t:main}]
We now fix a $k\in\mathbb{N}$ large enough that Proposition~\ref{p:fmuomegadecay} can be applied and $s_{\Phi}/3-o_{k}(1)\geq s_{\Phi}/4$, where the $o_{k}(1)$ term is as in the statement of this proposition. It follows from Proposition~\ref{p:fmuomegadecay} that for all $\xi \neq 0$, 
\begin{align*} 
|\widehat{F\mu}(\xi)|&\leq \int_{\Omega_{\xi}} |\widehat{F\tilde{\mu}}_{\omega}(\xi)| dP(\omega) + P(\Omega \setminus \Omega_{\xi})\\
	&\preceq (1+\|F\|_{\infty,1}+\|F\|_{\infty,1}^{-\epsilon}+\|F\|_{\infty,2}) (1+\|F\|_{\min,2}^{-s_{\Phi}})|\xi|^{ - \min\{ s_{\Phi}/4, \epsilon /3,0.2 \}} \\*
	&\phantom{--}+ (1+\|F\|_{\infty,1}^{-\epsilon})(1+\|F\|_{\min,2}^{-\beta})|\xi|^{-\beta'}\\
	&\preceq (1+\|F\|_{\infty,1}+\|F\|_{\infty,1}^{-\epsilon}+\|F\|_{\infty,2})\\*
	&\phantom{--} \times (1+\|F\|_{\min,2}^{-\max\{s_{\Phi},\beta\}})|\xi|^{ - \min\{ s_{\Phi}/4, \epsilon /3, 0.2, \beta' \}}. 
\end{align*} 
Taking $\kappa=\max\{\epsilon,s_{\Phi},\beta\}$, and $\eta=\min\{ s_{\Phi}/4, \epsilon /3, 0.2, \beta'\} $ completes the proof of the first part of Theorem~\ref{t:main}. 
Similarly, the second part of Theorem~\ref{t:main} follows from Proposition~\ref{p:onlyinsupport}. 
\end{proof}

\begin{remark}
Results for the measures $\mu_{\omega}$ may be of interest in their own right. 
By a straightforward application of the first Borel--Cantelli lemma, we see from Propositions~\ref{p:decayoutsidesparse} and~\ref{p:fmuomegadecay} that when $k$ is sufficiently large, for $P$-almost every $\omega$, for all $|\xi|$ sufficiently large (where `sufficiently large' can depend on $\omega$), $\mu_{\omega}$ will exhibit polynomial Fourier decay outside a sparse set of frequencies, and images of $\mu_{\omega}$ under maps with positive second derivative will exhibit polynomial Fourier decay. %
\end{remark}

\section{Future directions}
\label{Future directions}
In this section we ask several questions. 
The \emph{Fourier dimension} of a Borel measure $\mu$ on $\R^d$ is 
\begin{align*} \dim_{\mathrm F} \mu \coloneqq 2 \sup \{ \epsilon \geq 0 :& \mbox{ there exists } C_{\epsilon} > 0 \mbox{ such that } |\widehat{\mu}(\xi)| \leq C_{\epsilon} |\xi|^{-\epsilon}\\*
	& \mbox{ for all } \xi \neq 0 \}. 
	\end{align*}
The \emph{Fourier dimension} of a Borel set $X \subset \R^d$ is 
\[ \dim_{\mathrm F} X \coloneqq \sup\{ s \in [0,d] : \exists \mu \in \mathcal{M}(X) \mbox{ such that} \dim_{\mathrm F} \mu \geq s \},\]
where $\mathcal{M}(X)$ is the set of finite Borel measures with compact support contained in $X$. 
The conclusion of Theorem~\ref{thm:analyticthm} implies that the self-conformal measure and its support have positive Fourier dimension. It is natural to ask the following. 
	\begin{question}
	Let $\Phi = \{\varphi_a\colon [0,1] \to [0,1]\}_{a\in \A}$ be an IFS such that each $\varphi_{a}$ is analytic, and suppose that there exists $a$ such that $\varphi_{a}$ is not an affine map. 
	Can one obtain improved lower bounds on the Fourier dimensions of the attractor $X$ and self-conformal measures supported on $X$? 
	In particular, does $\dim_{\mathrm F} X = \dim_{\mathrm H} X$ always hold (i.e. is $X$ necessarily a Salem set)? 
	\end{question}
	If $X$ is indeed always a Salem set then this is likely to be difficult to prove. 	
		In a different direction, one can ask the following. 
	\begin{question}
	Consider an IFS of analytic maps on $[0,1]$, at least one of which is not affine. If $\mu$ is merely assumed to be a non-atomic Gibbs measure for a H\"older potential (or a quasi-Bernoulli measure) for the IFS, does $\mu$ necessarily have polynomial Fourier decay? 
	\end{question}
\begin{question}
Consider an IFS $\Phi=\{\varphi_a\colon [0,1] \to [0,1]\}_{a\in \A}$ whose contractions have weaker regularity than being analytic. Are there easily verifiable conditions for $\Phi$ under which the conclusion of Theorem~\ref{thm:analyticthm} still holds? 
\end{question}

Finally, one can ask if an appropriate analogue of Theorem~\ref{thm:analyticthm} holds in higher dimensions. Note that if the measure is contained in a proper subspace of $\R^d$ then there will be no Fourier decay in directions orthogonal to the subspace. 
\begin{question} 
If $\Phi$ is a conformal IFS on a subset of $\R^d$ for some $d \geq 2$, are there natural conditions on $\Phi$ which guarantee that every self-conformal measure has polynomial Fourier decay? 
\end{question} 
After the present paper appeared on arXiv, the first named author, Khalil and Sahlsten~\cite{BKS} used tools from Khalil~\cite{KhalilPreprintfourier} to establish rates of decay for the Fourier transform of a wide class of dynamically defined measures which includes certain self-conformal measures and certain Gibbs measures on $\R^d$. 
Moreover, Algom, Rodriguez Hertz and Wang~\cite{AHW4} proved polynomial Fourier decay for a class of self-conformal measures in the plane satisfying a nonlinearity condition. 
	
	\section*{Acknowledgements} 
	Both authors were financially supported by an EPSRC New Investigators Award (EP/W003880/1).
	We thank Amir Algom, Jonathan Bennett, Thomas Jordan and Tuomas Sahlsten for useful discussions. 
		We thank the authors of the independent work~\cite{AlgomPreprintnonlinear} for coordinating so that both papers were submitted to arXiv simultaneously. 
        We thank the anonymous referee for many helpful comments. 
	
	\bibliographystyle{plain}

\begin{thebibliography}{10}
		
		\bibitem{ABS} A. Algom, S. Baker, P. Shmerkin, \textit{On normal numbers and self-similar measures,} Adv. Math. \textbf{399} (2022), 108276.
		\bibitem{AlgomPreprintnonlinear} A. Algom, Y. Chang, M. Wu, Y.-L. Wu, \textit{Van der Corput and metric theorems for geometric progressions for self-similar measures,} arXiv:2401.01120.
		\bibitem{AHW}
		A. Algom, F. Rodriguez Hertz, Z. Wang, \textit{Pointwise normality and Fourier decay for self-conformal measures,} Adv. Math. \textbf{393} (2021), 108096. 
		
		\bibitem{AHW2} A. Algom, F. Rodriguez Hertz, Z. Wang, \textit{Logarithmic Fourier decay for self-conformal measures,} J. Lond. Math. Soc. \textbf{106} (2022), 1628--1661. 
		
		\bibitem{AHW3} A. Algom, F. Rodriguez Hertz, Z. Wang, \textit{Polynomial Fourier decay and a cocycle version of Dolgopyat’s method for self conformal measures,} arXiv:2306.01275v1.
        \bibitem{AHW4} A. Algom, F. Rodriguez Hertz, Z. Wang, \textit{Spectral gaps and Fourier decay for self-conformal measures in the plane,} arXiv:2407.11688.
		\bibitem{ABCY} D. Allen, S. Baker, S. Chow, H. Yu, \textit{A note on dyadic approximation in Cantor's set,} Indag. Math. (N.S.) \textbf{34} (2023), 190--197.
		\bibitem{ACY} D. Allen, S. Chow, H. Yu, \textit{Dyadic approximation in the middle-third Cantor set,} Selecta Math. (N.S.) \textbf{29} (2023), 11.
		\bibitem{ABY} A. Avila, J. Bochi, J.-C. Yoccoz, \textit{Uniformly hyperbolic finite-valued $SL(2;\mathbb{R})$-cocycles,} Comment. Math. Helv., 85(4), 813--884, 2010.
		\bibitem{Bak} S. Baker, \textit{Approximating elements of the middle third Cantor set with dyadic rationals,} Israel J. Math (to appear), arXiv:2203.12477. 
		\bibitem{BKS} S. Baker, O. Khalil, T. Sahlsten, \textit{Fourier decay from $L^2$-flattening,} arXiv:2407.16699. 
		\bibitem{BS} S. Baker, T. Sahlsten, \textit{Spectral gaps and Fourier dimension for self-conformal sets with overlaps}, arXiv:2306.01389v1.
		\bibitem{Bar} K. Bara\'{n}ski, \textit{Hausdorff dimension of the limit sets of some planar geometric constructions,} Adv. Math. \textbf{210} (2007), 215--245.
		\bibitem{BKPW} B. B\'{a}r\'{a}ny, A. K\"{a}enm\"{a}ki, A. Py\"{o}r\"{a}l\"{a}, M. Wu, \textit{Scaling limits of self-conformal measures,} arXiv:2308.11399.
		\bibitem{Bed} T. Bedford, \textit{Crinkly curves, Markov partitions and box dimensions in self-similar sets,} PhD dissertation, 1984.
		\bibitem{Bor} E. Borel, \textit{Les probabilités dénombrables et leurs applications arithmétiques,} Rendiconti del Circolo Matematico di Palermo, \textbf{27} (1909), 247--271.
		\bibitem{BD1} J. Bourgain, S. Dyatlov.
		 \textit{Fourier dimension and spectral gaps for hyperbolic surfaces,} Geom. Funct. Anal. \textbf{27} (2017), 744--771. 
		 \bibitem{BD2} J. Bourgain and S. Dyatlov. 
		 \textit{Spectral gaps without the pressure condition,} Ann. of Math. (2) \textbf{187} (2018), 825--867.

        \bibitem{Bug} Y. Bugeaud, \textit{Distribution modulo one and Diophantine approximation,} Cambridge Tracts in Math., 193
Cambridge University Press, Cambridge, 2012. xvi+300 pp.
		\bibitem{Cas} J.W.S. Cassels, \textit{On a problem of Steinhaus about normal numbers,} Colloq. Math. \textbf{7} (1959), 95--101. 
        \bibitem{Chang2017Fourier} Y. Chang, X. Gao, \textit{Fourier decay bound and differential images of self-similar measures,} arXiv:1710.07131.	
        \bibitem{DEL} H. Davenport, P. Erd\H{o}s, W.~J.~LeVeque, \textit{On Weyl's criterion for uniform distribution,} Michigan Math. J. \textbf{10} (1963), 311--314.

		\bibitem{DGW} Y. Dayan, A. Ganguly, B. Weiss, \textit{Random walks on tori and normal numbers in self similar sets,} Amer. J. Math. \textbf{146} (2024), 467--493.
		\bibitem{Dur} R. Durrett, \textit{Probability—theory and examples,}	Camb. Ser. Stat. Probab. Math., \textbf{49}
		Cambridge University Press, Cambridge, 2019. xii+419 pp.
		\bibitem{Dyatlov} S. Dyatlov, \textit{An introduction to fractal uncertainty principle,} J. Math. Phys. \textbf{60} (2019), 081505.
		
		\bibitem{DJ}
		S. Dyatlov, L. Jin, \textit{Semiclassical measures on hyperbolic surfaces have full support,} Acta Math. \textbf{220} (2018), 297--339.
		\bibitem{DJN} S. Dyatlov, L. Jin, S. Nonnenmacher,
		\textit{Control of eigenfunctions on surfaces of variable curvature,} J. Amer. Math. Soc. \textbf{35} (2022), 361--465. 
        \bibitem{DZ} S. Dyatlov, J. Zahl, \textit{Spectral gaps, additive energy, and a fractal uncertainty principle}, Geom. Funct. Anal. \textbf{26} (2016), 1011--1094. 
        
		\bibitem{Erdos2} P. Erd\H{o}s, \textit{On the smoothness properties of a family of Bernoulli convolutions,} Amer. J. Math. \textbf{62} (1940), 180--186. 
		\bibitem{FFK} K. J. Falconer, J. M. Fraser, A. K\"aenm\"aki. \textit{Minkowski dimension for measures,} Proc. Amer. Math. Soc. \textbf{151} (2023), 779--794. 
				\bibitem{FL} D.-J. Feng, K.-S. Lau, \textit{Multifractal formalism for self-similar measures with weak separation condition,}
		J. Math. Pures Appl. \textbf{92} (2009), 407--428.
		\bibitem{GSSY} D. Galicer, S. Saglietti, P. Shmerkin, A. Yavicoli, \textit{$L^q$ dimensions and projections of random measures,} Nonlinearity \textbf{29} (2016), 2609--2640.
		\bibitem{GL} S. Lalley, D. Gatzouras, \textit{Hausdorff and box dimensions of certain self-affine fractals,} Indiana Univ. Math. J. \textbf{41} (1992), 533--568. 
		\bibitem{Hoc} M. Hochman, \textit{Dimension theory of self-similar sets and measures,} Proceedings of the ICM 2018 (eds. B. Sirakov, P. N. de Souza, and M. Viana), vol. 3, pp. 1943--1966. World Scientific, 2019.
		\bibitem{HocShm} M. Hochman, P. Shmerkin, \textit{Equidistribution from fractal measures,}
		Invent. Math. \textbf{202} (2015), 427--479.
	    \bibitem{Hoe} W. Hoeffding, \textit{Probability inequalities for sums of bounded random variables,} J. Amer. Statist. Assoc. \textbf{58} (1963) 13--30.
        \bibitem{Hos} B. Host, \textit{Nombres normaux, entropie, translations,} Israel J. Math. \textbf{91} (1995), 419--428.
		\bibitem{Hut} J. Hutchinson, \textit{Fractals and self-similarity}, Indiana Univ. Math. J. \textbf{30} (1981), 713--747.	
		\bibitem{JordanSahlsten}
		T. Jordan, T. Sahlsten, \textit{Fourier transforms of Gibbs measures for the Gauss map,} Math. Ann. \textbf{364} (2016), 983--1023. 
		\bibitem{KaeOrp} A. K\"{a}enm\"{a}ki, T. Orponen, \textit{Absolute continuity in families of parametrised non-homogeneous self-similar measures,}	J. Fractal Geom. \textbf{10} (2023), 169--207. 
		\bibitem{Kah} J.-P. Kahane,  \textit{Sur la distribution de certaines s\'{e}ries al\'{e}atoires,} In Colloque de Th\'{e}orie des Nombres (Univ. Bordeaux, Bordeaux, 1969), pages 119--122. Bull. Soc. Math. France, M\'{e}m. No. 25, Soc. Math. France Paris, 1971.
		\bibitem{Kau2}  R. Kaufman, \textit{Continued fractions and Fourier transforms,} Mathematika \textbf{27} (1980), 262--267. 
		\bibitem{Kau} R. Kaufman, \textit{On Bernoulli convolutions,} Conference in modern analysis and probability (New Haven, Conn., 1982), 217--222.
		Contemp. Math., 26
		American Mathematical Society, Providence, RI, 1984. 
		
		\bibitem{KL} A. S. Kechris, A. Louveau, \textit{Descriptive set theory and the structure of sets of uniqueness,} London Mathematical Society Lecture Note Series, Vol. 128, Cambridge University Press, Cambridge, 1987.
		
		\bibitem{KhalilPreprintfourier} O. Khalil, \textit{Exponential mixing via additive combinatorics,} arXiv:2305.00527v3. 

        \bibitem{Klenke} A. Klenke, \textit{Probability Theory,} Springer, Berlin, 2008. 

        \bibitem{Lindenstrauss} E. Lindenstrauss, \textit{$p$-adic foliation and equidistribution,} Israel J. Math. \textbf{122} (2001), 29--42. 
		
		\bibitem{JialunSahlsten1} J. Li, T. Sahlsten, \textit{Trigonometric series and self-similar sets,} J. Eur. Math. Soc. \textbf{24} (2022), 341--368.
		
		\bibitem{JialunSahlsten2} J. Li, T. Sahlsten, \textit{Fourier transform of self-affine measures,} Adv. Math. \textbf{374} (2020), 107349.
		
		\bibitem{MU1} R.~D.~Mauldin, M.~Urba\'nski, \textit{Dimensions and measures in infinite iterated function systems,} Proc. Lond. Math. Soc. (3) \textbf{73} (1996), 105--154. 
		
		\bibitem{Mitsis2002restriction} T. Mitsis, \textit{A Stein--Tomas restriction theorem for general measures}, Publ. Math. Debrecen \textbf{60} (2002), 89--99. 
		
		\bibitem{Mockenhaupt2000restriction} G. Mockenhaupt, \textit{Salem sets and restriction properties of Fourier transforms}, Geom. Funct. Anal. \textbf{10} (2000), 1579--1587. 
		
		\bibitem{MC} C. McMullen, \textit{The Hausdorff dimension of general Sierpiński carpets,} Nagoya Math. J. \textbf{96} (1984), 1--9.

        \bibitem{MO} C. Mosquera, A. Olivo, \textit{Fourier decay behavior of homogeneous self-similar measures on the complex plane}, J. Fractal Geom. \textbf{10} (2023), 43--60. 
        
		\bibitem{MS}
		C. Mosquera, P. Shmerkin, \textit{Self-similar measures: asymptotic bounds for the dimension and Fourier decay of smooth images,} Ann. Acad. Sci. Fenn. Math. \textbf{43} (2018), 823--834.
		\bibitem{Phi} W. Philipp, \textit{Some metrical theorems in number theory,} Pacific J. Math. 20 (1967), 109--127.
		\bibitem{PVZZ} A. Pollington, S. Velani, A. Zafeiropoulos, E. Zorin, \textit{Inhomogeneous Diophantine approximation on $M_{0}$-sets with restricted denominators,} Int. Math. Res. Not. IMRN (2022), 8571--8643.
		\bibitem{QR} M. Queff\'{e}lec, O. Ramar\'{e}, \textit{Analyse de Fourier des fractions continues a quotients restreints,} Enseign. Math \textbf{49} (2003), 335--356.

		\bibitem{SSS} S. Saglietti, P. Shmerkin, B. Solomyak, \textit{Absolute continuity of non-homogeneous self-similar measures,} Adv. Math. \textbf{335} (2018), 60--110. 
		\bibitem{SahlstenSurvey} T.~Sahlsten, \textit{Fourier transforms and iterated function systems,} arXiv:2311.00585. 
		\bibitem{SS} T. Sahlsten, C. Stevens, \textit{Fourier transform and expanding maps on Cantor sets}, Amer. J. Math. \textbf{146} (2024), 945--982. 
        
		\bibitem{Salem} R. Salem, \textit{Sets of uniqueness and sets of multiplicity,} Trans. Amer. Math. Soc. \textbf{54} (1943), 218--228. Corrected Trans. Amer. Math. Soc. \textbf{63} (1948), 595--598.
		
		\bibitem{Secelean2002countable} N.~Secelean, \textit{The invariant measure of an countable iterated function system,} Seminarberichte aus dem Fachbereich Mathematik der Fernuniversität Hagen \textbf{73} (2002), 3--10.
		\bibitem{Shmerkinsurvey} P. Shmerkin, \textit{Slices and distances: on two problems of Furstenberg and Falconer,} Proceedings of the ICM 2022 (eds. D. Beliaev and S. Smirnov), vol. 4, pp. 3266--3290, EMS Press, 2023. 
		\bibitem{Stein1993harmonic} E. M. Stein, \textit{Harmonic analysis: real variable methods, orthogonality, and oscillatory integrals}, volume 43 of \textit{Princeton Mathematical Series}, Princeton University Press, Princeton, NJ, 1993. 

		\bibitem{SolSpi} B. Solomyak, A. \'Spiewak, \textit{Absolute continuity of self-similar measures on the plane}, Indiana Univ. Math. J. (to appear), arXiv:2301.10620. 
		\bibitem{Tsujii} M. Tsujii, \textit{On the Fourier transforms of self-similar measures,} Dyn. Syst. \textbf{30} (2015), 468--484. 

		\bibitem{Varjusurvey2} P. Varj\'u, \textit{Self-similar sets and measures on the line,}  Proceedings of the ICM 2022 (eds. D. Beliaev and S. Smirnov), vol. 5, pp. 3610--3634, EMS Press, 2023.  

		\bibitem{Wormell2022} C. L. Wormell, \textit{Conditional mixing in deterministic chaos,} Ergodic Theory Dynam. Systems \textbf{44} (2024), 1693--1723.  
		\bibitem{Yoccoz} J.-C. Yoccoz, \textit{Some questions and remarks about $SL(2;\mathbb{R})$ cocycles,} In Modern dynamical
		systems and applications, 447--458. Cambridge Univ. Press, Cambridge, 2004.
	\end{thebibliography}
	
	\hspace{1cm}
	
	\footnotesize
	{\parindent0pt
		\textsc{Department of Mathematical Sciences, Loughborough University, Loughborough, LE11 3TU, United Kingdom.}\medbreak
		\par\nopagebreak
		\textit{Email:} \texttt{simonbaker412@gmail.com}\medbreak
		\par\nopagebreak
		\textit{Email:} \texttt{banajimath@gmail.com}
	}
	
\end{document}